\documentclass[12pt]{amsart}
\usepackage{amssymb, amsfonts}

\newcommand{\x}{\times}
\newcommand{\<}{\langle}
\renewcommand{\>}{\rangle}

\renewcommand{\a}{\alpha}
\renewcommand{\o}{\omega}
\renewcommand{\b}{\beta}
\renewcommand{\d}{\delta}
\newcommand{\D}{\Delta}
\newcommand{\e}{\varepsilon}
\newcommand{\g}{\gamma}
\renewcommand{\l}{\lambda}
\renewcommand{\L}{\Lambda}
\newcommand{\var}{\varphi}
\newcommand{\s}{\sigma}
\renewcommand{\th}{\theta}
\renewcommand{\O}{\Omega}
\newcommand{\z}{\zeta}

\renewcommand{\i}{\infty}

\newcommand{\Lip}{\mathop{\mathrm{Lip}}\nolimits}
\newcommand{\Aut}{\mathop{\mathrm{Aut}}\nolimits}
\newcommand{\dist}{\mathop{\mathrm{dist}}\nolimits}
\newcommand{\prox}{\mathop{\mathrm{prox}}\nolimits}
\newcommand{\tr}{\mathop{\mathrm{tr}}\nolimits}

\newcommand{\diag}{\mathop{\mathrm{diag}}\nolimits}

\newcommand{\cB}{{\mathcal B}}
\newcommand{\cE}{{\mathcal E}}
\newcommand{\cH}{{\mathcal H}}
\newcommand{\cK}{{\mathcal K}}
\newcommand{\cL}{{\mathcal L}}
\newcommand{\cO}{{\mathcal O}}
\newcommand{\cM}{{\mathcal M}}
\newcommand{\cP}{{\mathcal P}}

\newcommand{\cS}{{\mathcal S}}

\newcommand{\bZ}{{\mathbb Z}}
\newcommand{\bR}{{\mathbb R}}
\newcommand{\bC}{{\mathbb C}}
\newcommand{\bT}{{\mathbb T}}

\newcommand{\fG}{{\mathfrak g}}

\newtheorem{theorem}{Theorem}[section]
 
\newtheorem{proposition}[theorem]{Proposition} 

\newtheorem{mainlemma}[theorem]{Main Lemma} 
\newtheorem{lemma}[theorem]{Lemma} 

\theoremstyle{definition}   
\newtheorem{definition}[theorem]{Definition} 
\newtheorem{example}[theorem]{Example} 
\newtheorem{notation}[theorem]{Notation}

\begin{document}

\title[Leibniz seminorms]{Leibniz seminorms for \\
``Matrix algebras converge to the sphere''}
\author{Marc A. Rieffel}
\address{Department of Mathematics \\
University of California \\ Berkeley, CA 94720-3840}
\email{rieffel@math.berkeley.edu}
\thanks{The research reported here was
supported in part by National Science Foundation grant DMS-0500501.}

\begin{abstract}
In an earlier paper of mine relating vector bundles and
Gromov--Hausdorff distance for ordinary compact metric spaces, it was
crucial that the Lipschitz seminorms from the metrics satisfy a strong
Leibniz property.  In the present paper, for the now non-commutative
situation of matrix algebras converging to the sphere (or to other
spaces) for quantum Gromov--Hausdorff distance, we show how to
construct suitable seminorms that also satisfy the strong Leibniz
property.  This is in preparation for making precise certain
statements in the literature of high-energy physics concerning
``vector bundles'' over matrix algebras that ``correspond'' to
monopole bundles over the sphere.  We show that a fairly general
source of seminorms that satisfy the strong Leibniz property consists
of derivations into normed bimodules.  For matrix algebras our main
technical tools are coherent states and Berezin symbols.
\end{abstract}

\maketitle
\allowdisplaybreaks

\section*{Introduction}

In a previous paper \cite{R7} I showed how to give a precise meaning
to statements in the literature of high-energy physics and string
theory of the kind ``Matrix algebras converge to the sphere''.  (See
\cite{R7} for numerous references to the relevant physics literature.)
I did this by introducing the concept of ``compact quantum metric
spaces'', in which the metric data is given by a seminorm on the
non-commutative ``algebra of functions''.  This seminorm plays the
role of the usual Lipschitz seminorm on the algebra of continuous
functions on an ordinary compact metric space.  However, I was
somewhat puzzled by the fact that I needed virtually no algebraic
conditions on the seminorm, only an important analytic condition.  But
when I later began trying to give precise meaning to further
statements in the physics literature of the kind ``here are the vector
bundles over the matrix algebras that correspond to the monopole
bundles over the sphere'' (see \cite{R17} for many references), I
found that for ordinary metric spaces a strong form of the Leibniz
inequality for the seminorm played a crucial role \cite{R17}.  (See,
for example, the proof of proposition~$2.3$ of \cite{R17}.)  However,
on returning to the non-commutative case of matrix algebras converging
to the sphere (or to other spaces), for some time I did not see how to
construct useful seminorms that brought the matrix algebras and sphere
close together while also having the strong Leibniz property.  The
main purpose of this paper is to show how to construct such seminorms.
As in the earlier paper \cite{R7}, the setting is that of coadjoint
orbits of compact semisimple Lie groups, of which the 2-sphere is the
simplest example.  The main technical tools continue to be coherent
states and Berezin symbols.

In the first four sections of this paper we show that a fairly general
setting for obtaining seminorms that possess the strong Leibniz
property that we need consists of derivations into normed bimodules,
and we examine various aspects of this topic.  The strong Leibniz
property for a seminorm $L$ on a normed unital algebra $A$ consists of
the usual Leibniz inequality together with the inequality
\[
L(a^{-1}) \le \|a^{-1}\|^2L(a)
\]
whenever $a$ is invertible in $A$.  I have not seen this latter
inequality discussed in the literature. In Section 4 we put together the 
various conditions that we have found to be important, and there-by
give a tentative definition for a ``compact $C^*$-metric space''.

In Section~\ref{sec5} we examine the use of seminorms with the strong
Leibniz property in connection with quantum Gromov--Hausdorff
distance.  (I expect that many of the ideas and techniques developed
in this paper will apply to many other classes of examples beyond
``Matrix algebras converge to the sphere''.)  In Section~\ref{sec6} we
extend to the case of strongly Leibniz seminorms the construction
technique introduced in \cite{R6} that we called ``bridges''.
Sections~\ref{sec7} and \ref{sec8} contain those pieces of our
development that can be carried out for certain homogeneous spaces of
any compact group (including finite ones).  Section~\ref{sec9} gives
the statement of our main theorem for coadjoint orbits, while
Sections~\ref{sec10} through \ref{sec13} contain the detailed
technical development needed to prove our main theorem.  Finally, in
Section~\ref{sec14} we relate our results to other variants of quantum
Gromov--Hausdorff distance that have been developed by David~Kerr,
Hanfeng~Li, and Wei~Wu \cite{Krr,KrL,Lih2,Lih3,Wuw1,Wuw2,Wuw3}.

We can describe our basic setup and our main theorem somewhat more
specifically as follows, where definitions for various terms are given
in later sections.  Let $G$ be a compact semisimple Lie group, let
$(U,\cH)$ be an irreducible unitary representation of $G$, and let $P$
be the rank-one projection along a highest weight vector for
$(U,\cH)$.  Let $\a$ be the action of $G$ on $\cL(\cH)$ by conjugation
by $U$, and let $H$ be the $\a$-stability group of $P$.  Let $A =
C(G/H)$.  Let $\o$ be the highest weight for $U$, and for each $n \in
\bZ_{>0}$ let $(U^n,\cH^n)$ be the irreducible representation of $G$
of highest weight $n\o$.  Let $\a$ also denote the action of $G$ on
$B^n = \cL(\cH^n)$ by conjugation by $U^n$.

Choose on $G$ a continuous length-function $\ell$.  Then $\ell$ and
the translation action of $G$ on $A$, as well as the actions $\a$ of
$G$ on each $B^n$, determine seminorms $L_A$ on $A$ and $L_{B^n}$ on
$B^n$ that make $(A,L_A)$ and each $(B^n,L_{B^n})$ into compact
$C^*$-metric spaces.

\medskip
\noindent
{\bf Main Theorem} (sketchy statement of Theorem 9.1).  {\em 
For any $\e > 0$ there exists an $N$ such that for any $n \ge N$ we
can explicitly construct a strongly Leibniz seminorm, $L_n$, on $A
\oplus B^n$ making $A \oplus B^n$ into a compact $C^*$-metric space,
such that the quotients of $L_n$ on $A$ and $B^n$ are $L_A$ and
$L_{B^n}$, and for which the quantum Gromov-Hausdorff distance 
between $A$ and $B^n$
is no greater than $\e$.
}

\medskip
I plan to apply the results of this paper in a future paper to discuss
vector bundles over non-commutative spaces (e.g., monopole bundles),
along the lines used for ordinary spaces in \cite{R17}.

I developed part of the material presented here during a ten-week
visit at the Isaac Newton Institute in Cambridge, England, in the Fall
of 2006.  I am very appreciative of the stimulating and enjoyable
conditions provided by the Isaac Newton Institute.

I am grateful to Hanfeng Li for some important comments 
on the first version of this paper, which led to some substantial improvements
given in the present version.


\section{Strongly Leibniz seminorms}
\label{sec1}

From my investigation of the relation between vector bundles on
compact metric spaces that are close together, both for ordinary
spaces \cite{R17} and for non-commutative spaces (a continuing
investigation), I have found that the following properties are very
important when considering the seminorms that play the role of the
Lipschitz seminorms of ordinary metric spaces.  Unless the contrary is
stated, we allow our seminorms to take the value $+\i$, but we require
that they take value $0$ at $0$.  We use the usual conventions for
calculating with $+\i$. The following definition is close to definition 2.1
of \cite{R17}.

\begin{definition}
\label{def1.1}
Let $A$ be a normed unital algebra over $\bR$ or $\bC$, and let $L$ be
a seminorm on $A$.  We say that:
\begin{itemize}
\item[1)] $L$ is {\em Leibniz} if it satisfies the inequality
\[
L(ab) \le L(a)\|b\| + \|a\|L(b)
\]
for all $a,b \in A$.
\item[2)] $L$ is {\em strongly Leibniz} if it is Leibniz, and $L(1) =
0$, and if for any $a \in A$ that has an inverse in $A$, we have
\[
L(a^{-1}) \le \|a^{-1}\|^2L(a).
\]
\item[3)] $L$ is {\em finite} if $L(a) < \i$ for all $a \in A$.
\item[4)] $L$ is {\em semifinite} if $\{a: L(a) < \i\}$ is norm-dense
in $A$.
\item[5)] $L$ is {\em continuous} if it is norm-continuous.
\item[6)] $L$ is {\em lower semicontinuous} if for one $r \in
\bR_{>0}$, hence for all $r > 0$, the set
\[
\{a \in A: L(a) \le r\}
\]
is norm-closed in $A$.
\end{itemize}
If, furthermore, $A$ is a $*$-normed algebra (i.e., has an isometric
involution), then we define $L^*$ by $L^*(a) = L(a^*)$ for $a \in A$.
We then say that $L$ is a {\em $*$-seminorm} if $L = L^*$.
\end{definition}

The proof of the following proposition is straightforward.

\begin{proposition}
\label{prop1.2}
Let $A$ be a unital normed algebra.
\begin{itemize}
\item[i.] Let $L$ be a seminorm on $A$ and let $r \in \bR^+$.  If $L$
satisfies one of the properties 1--6 above then $rL$ satisfies that
same property.
\item[ii.] Let $L_1$ and $L_2$ be two seminorms on $A$.  If they are
both Leibniz, or strongly Leibniz, or finite, or continuous, or lower
semicontinuous, then so is $L_1 + L_2$.
\item[iii.] Let $\{L_{\a}\}$ be a family of seminorms on $A$, possibly
infinite, and let $L$ be the supremum of this family.  (I.e., $L =
\bigvee_{\a} L_{\a}$, defined by $L(a) = \sup_{\a} \{L_{\a}(a)\}$.
For two seminorms, $L$ and $L'$, we will denote their maximum by $L
\vee L'$.)  Then $L$ is a seminorm on $A$, and if each $L_{\a}$ is
Leibniz, or strongly Leibniz, or lower semicontinuous, then so is $L$.
\item[iv.] If $A$ is a $*$-normed algebra and if $L$ satisfies one of
the properties 1--6 above, then $L^*$ satisfies that same property.
\end{itemize}
\end{proposition}

I have seen no discussion of the strong Leibniz property in the
literature.  I do not know of an example of a finite Leibniz seminorm
which does not satisfy the inequality for $L(a^{-1})$ in the
definition of ``strongly Leibniz''.  But if we allow the value $+\i$
then examples can be constructed in the following way.  Let $A$ be a
unital normed algebra and let $B$ be a unital subalgebra of $A$.  Let
$L_0$ be a finite Leibniz seminorm on $B$.  Define a Leibniz seminorm,
$L$, on $A$ by $L(a) = L_0(a)$ if $a \in B$ and $L(a) = +\i$
otherwise.  If $B$ contains an element that is invertible in $A$ but
not in $B$ then $L$ is not strongly Leibniz.  For example, let $A =
C([0,1])$ and let $B$ be its subalgebra of polynomial functions, with
$L_0(f) = \|f'\|$.  (This example is not lower semicontinuous.)

Let $A^f = \{a: L(a) < \i\}$.  It is clear that if $L$ is Leibniz then
$A^f$ is a subalgebra of $A$.  If $L$ is in fact strongly Leibniz and
$a \in A^f$, then clearly $a$ is invertible in $A^f$ if and only if it
is invertible in $A$.  It follows that for any $a \in A^f$ the
spectrum of $a$ in $A^f$ will be the same as its spectrum in $A$.  In
stupid examples we may have $1 \notin A^f$, but with that understood,
we see that:

\begin{proposition}
\label{prop1.3}
If $L$ is strongly Leibniz then $A^f$ is a spectrally stable
subalgebra of $A$.
\end{proposition}

The importance of this proposition will be seen in Section~\ref{sec3}.
We also remark that if $A$ has an involution and if $L$ is a Leibniz
seminorm that is also a $*$-seminorm, then $A^f$ is a $*$-subalgebra
of $A$.

Simple arguments prove the following two propositions.

\begin{proposition}
\label{prop1.4}
Let $A$ be a normed unital algebra, and let $L$ be a seminorm on $A$.
Let $B$ be a unital subalgebra of $A$, equipped with the norm from
$A$.  If $L$ is Leibniz, or strongly Leibniz, or finite, or
continuous, or lower semicontinuous, then so is the restriction of $L$
to $B$ as a seminorm on $B$.  (But if $L$ is semifinite, its
restriction to $B$ need not be semifinite.)
\end{proposition}

\begin{proposition}
\label{prop1.5}
Let $A$ be a $*$-normed unital algebra and let $L$ be a seminorm on $A$. Let
${\tilde L} = L \vee (L^*)$.  Then ${\tilde L}$ is a $*$-seminorm.
 If $L$ is Leibniz, or strongly Leibniz, or finite, or
continuous, or lower semicontinuous, then so is $\tilde L$.
(But if $L$ is semifinite, $\tilde L$ need not be semifinite.)
\end{proposition}

So in this way we can usually arrange to work with $*$-seminorms when
dealing with $*$-algebras.

Here is another way to combine seminorms:

\begin{proposition}
\label{prop1.6}
Let $L_1,\dots,L_n$ be seminorms on a normed unital algebra $A$, and
let $\|\cdot\|_0$ be a norm on $\bR^n$ with the property that if
$(r_j),(s_j) \in \bR^n$, and if $0 \le r_j \le s_j$ for all $j$, $1
\le j \le n$, then $\|(r_j)\|_0 \le \|(s_j)\|_0$.  Define a seminorm,
$N$, on $A$ by $N(a) = \|(L_j(a))\|_0$, with the evident meaning if
$L_j(a)= \i$ for some $j$.  If each $L_j$ satisfies a particular one
of properties $1$, $2$, $3$, $5$, $6$ of Definition~{\em \ref{def1.1}}
then $N$ satisfies that property too.
\end{proposition}

\begin{proof}
If each $L_j$ is Leibniz, then
\begin{align*}
N(ab) &= \|(L_j(ab))\|_0 \le \|(L_j(a)\|b\| + \|a\|L_j(b))\|_0 \\
&\le \|(L_j(a))\|_0\|b\| + \|a\|\|(L_j(b))\|_0 = N(a)\|b\| +
\|a\|N(b),
\end{align*}
and if each $L_j$ is strongly Leibniz, then also
\[
N(a^{-1}) = \|(L_j(a^{-1}))\|_0 \le \|(\|a^{-1}\|^2L_j(a))\|_0 =
\|a^{-1}\|^2N(a).
\]
It is clear that if each $L_j$ is finite, or continuous, then so is
$N$.

Suppose instead that each $L_j$ is lower semicontinuous.  Let $(a_m)$
be a sequence in $A$ which converges in norm to $a \in A$, and suppose
that there is a constant, $K$, such that $N(a_m) \le K$ for each $m$.
For each $m$ let $p^m = (L_j(a_m)) \in \bR^n$, so that $\|p^m\|_0 \le
K$.  Since the $K$-ball of $\bR^n$ for $\|\cdot\|_0$ is compact, we
can pass to a convergent subsequence, so we can assume that the
sequence $\{p^m\}$ converges to a vector, $p$, in $\bR^n$ such 
that $\|p\|_0 \leq K$ and whose entries are non-negative.  Let $\e >
0$ be given.  Then there is an integer $m_{\e}$ such that if $m \ge
m_{\e}$ then $L_j(a_m) \le p_j + \e$ for each $j$.  Since each $L_j$
is lower semicontinuous, it follows that $L_j(a) \le p_j + \e$ for
each $j$.  Then $N(a) = \|(L_j(a))\|_0 \le \|(p_j+\e)\|_0 \le \|p\|_0
+ \e\|(1,\dots,1)\|_0$.  Thus $N(a) \le K$ since $\|p\|_0 \le K$ and
$\e$ is arbitrary.
\end{proof}


\section{General sources of strongly Leibniz seminorms}
\label{sec2}

We will now examine general methods for constructing strongly Leibniz
seminorms.  We recall first \cite{GVF} that a {\em first-order
differential calculus} over a unital algebra $A$ is a pair $(\O,d)$
consisting of a bimodule $\O$ over $A$ and a derivation $d$ from $A$
into $\O$, that is, a linear map from $A$ into $\O$ such that
\[
d(ab) = (da)b + a(db)
\]
for all $a,b \in \O$.  (We will always assume that our bimodules are
such that $1_A$ acts as the identity operator on both left and right.)
It is common to assume that $\O$ is generated as a bimodule by the
range of $d$, but we will not need to impose this requirement, though
it can always be arranged by replacing $\O$ by its sub-bimodule
generated by the range of $d$.

Suppose now that $A$ is a normed unital algebra (with $\|1_A\| = 1$),
and that $(\O,d)$ is a first-order differential calculus for $A$.
Assume further that $\O$ is equipped with a norm that makes it into a
normed $A$-bimodule, that is,
\[
\|a\o b\|_{\O} \le \|a\|\|\o\|_{\O}\|b\|
\]
for all $a,b \in A$ and $\o \in \O$.  We will then say that
$(\O,d,\|\cdot\|_{\O})$ is a {\em normed first-order differential
calculus}.  We do not require that $d$ be continuous for the norms on
$A$ and $\O$.  We define $L$ on $A$ by
\[
L(a) = \|da\|_{\O}.
\]
Notice that $L$ is finite, and that $L$ is continuous if $d$ is.

\begin{proposition}
\label{prop2.1}
Let $L$ on $A$ be defined as above for a normed first-order
differential calculus.  Then $L$ is strongly Leibniz.
\end{proposition}

\begin{proof}
That $L$ is Leibniz follows immediately from the definitions of a
derivation and of a normed bimodule.  To see that $L$ is strongly
Leibniz, notice first that from the definition of a derivation we
obtain $d(1_A) = 0$, so that $L(1_A) = 0$.  Suppose now that $a$ is an
invertible element of $A$.  Then
\[
0 = d(1_A) = d(aa^{-1}) = (da)a^{-1} + a(d(a^{-1})).
\]
Thus
\[
d(a^{-1}) = -a^{-1}(da)a^{-1}.
\]
On taking the norm we see that $L(a^{-1}) \le \|a^{-1}\|^2L(a)$.
\end{proof}

We remark that no effective characterization seems to be known as
to which Leibniz
seminorms come from normed first-order differential calculi (or
of inner ones) as above.
They fall within the scope of the ``flat'' differential seminorms defined
in definition 4.3 of \cite{BlC}, and for which equivalent conditions are
given in theorem 4.4 of \cite{BlC}. Necessary conditions for a
differential seminorm to be flat are given immediately after
definition 4.3 and in proposition 4.7 of \cite{BlC}. For Leibniz
seminorms we see above that a further necessary condition
is that of being strongly Leibniz.

Let us now give some examples.

\begin{example}
\label{exam2.2}
Let $(X,\rho)$ be a compact metric space.  For given $x_0,x_1 \in X$
with $x_0 \ne x_1$ let $\O_{x_0,x_1}$ be $\bR$ or $\bC$ according to
whether $A = C(X)$ is over $\bR$ or $\bC$, and define actions of $A$
on $\O_{x_0,x_1}$ by
\[
f \cdot \o = f(x_0)\o, \quad w \cdot f = \o f(x_1).
\]
Define $d$ by
\[
df = (f(x_1) - f(x_0))/\rho(x_1,x_0).
\]
It is easily checked that $(\O_{x_0,x_1},d)$ is a first-order
differential calculus over $A$.  Give $A = C(X)$ its supremum norm,
$\|\cdot\|_{\i}$, and give $\O_{x_0,x_1}$ the usual norm on $\bR$ or
$\bC$.  Then $\O_{x_0,x_1}$ is a normed $A$-bimodule.  Clearly $d$ is
continuous.  We set
\[
L_{x_0,x_1}(f) = \|df\| = |f(x_1) - f(x_0)|/\rho(x_1,x_0).
\]
Then from Proposition~\ref{prop2.1} it follows that $L_{x_0,x_1}$ is
strongly Leibniz (and continuous).

Now let $L$ be the supremum of the $L_{x_0,x_1}$ over all pairs
$(x_0,x_1)$ with $x_1 \ne x_0$.  We obtain in this way the usual
Lipschitz seminorm, $L^\rho$, on $C(X)$.  
From Proposition~{\ref{prop1.2}} it follows that $L^\rho$ is strongly
Leibniz and lower semicontinuous.  Of course $L^\rho$ is not
continuous in general.  But $L^\rho$ is semifinite, since it is finite
on the functions $f_{x_0}(x) = \rho(x,x_0)$, and these already
generate a dense subalgebra, as seen by means of the
Stone--Weierstrass theorem.

This example can be recast in a quite familiar form as follows.  Let
$Z = (X \x X)\setminus \D$ where $\D$ is the diagonal of $X \x X$.
Thus $Z$ is a locally compact space.  Let $\O = C_b(Z)$, the linear
space of bounded continuous functions on $Z$ with the supremum norm.
Then $\O$ is a normed $C(X)$-bimodule for the actions
\[
(f\o)(x_0,x_1) = f(x_0)\o(x_0,x_1),\quad (\o f)(x_0,x_1) =
\o(x_0,x_1)f(x_1).
\]
Let $A$ denote the subalgebra of $C(X)$ consisting of the Lipschitz
functions, and define a derivation $d$ from $A$ to $C(Z)$ by
\[
(df)(x_0,x_1) = (f(x_1) - f(x_0))/\rho(x_1,x_0).
\]
Then the usual Lipschitz seminorm is given by $L^\rho(f) =
\|df\|_{\i}$.  Alternatively, let $\O = C(Z)$, the space of all
continuous, possibly unbounded, functions on $Z$, as a $C(X)$-bimodule
in the above way.  Then $d$ can be defined on all of $C(X)$ by the
above formula.  We can now consider the supremum norm on $C(Z)$,
taking value $+\i$ on unbounded functions (so a bit beyond our
definitions above), and again set $L^\rho(f) = \|df\|_{\i}$.
\end{example}

\begin{example}
\label{exam2.3}
Let us now consider some examples in which the normed unital algebra
$A$ may be non-commutative.  If $\O$ is an $A$-bimodule, one always
has the corresponding inner derivations.  That is, if $\o \in \O$ we
can set $d^{\o}(a) = \o a - a\o$.  If $\O$ is a normed $A$-bimodule
then $d^{\o}$ is continuous, with $\|d^{\o}\| \le 2\|\o\|$.  The
corresponding seminorm, $L^{\o}$, defined by $L^{\o}(a) =
\|d^{\o}(a)\|$, is then a continuous strongly Leibniz seminorm.

Suppose now that $B$ is a unital normed algebra and that $\pi$ is a
unital homomorphism from $A$ into $B$.  Then we can view $B$ as a
bimodule over $A$ in the evident way, and obtain inner derivations and
corresponding strongly Leibniz seminorms, which are continuous if
$\pi$ is.
\end{example}

\begin{example}
\label{exam2.4}
Suppose now that $\pi$ is a non-degenerate representation of $A$ as
operators on a normed space $X$, so that $\pi$ can be viewed as a
unital homomorphism from $A$ into $B(X)$, the algebra of bounded
operators on $X$.  Then $B(X)$ can be viewed in the evident way as a
bimodule over $A$, and any element, $D$, of $B(X)$ determines an inner
derivation, and corresponding seminorm
\[
L(a) = \|D\pi(a) - \pi(a)D\| = \|[D,\pi(a)]\|.
\]

More generally, if one has two representations, $\pi^1$ and $\pi^2$ of
$A$ on $X$, then one can view $B(X)$ as an $A$-bimodule via
\[
a \cdot T \cdot b = \pi^1(a)T\pi^2(a),
\]
and again any element, $D$, of $B(X)$ will determine an inner
derivation.  (The twisted commutators in equation $2.4$ and lemma
$2.2$ of \cite{CnM} fit into this view, except that there $D$ is
usually an unbounded operator.)  Alternately one can assemble $\pi^1$
and $\pi^2$ into one representation on $X \oplus X$, and use the
operator $\begin{pmatrix} 0 & D \\ D & 0 \end{pmatrix}$ on $X \oplus
X$.

As an important particular case, for $X$ we can take $A$ itself and
let $\pi$ be the left-regular representation of $A$ on itself.  As
element of $B(X)$ we can take an isometric algebra automorphism, $\a$,
of $A$.  Then
\begin{align*}
(\a \circ \pi(a) - \pi(a) \circ \a)(b) &= \a(ab) - a\a(b) \\
&= (\a(a) - a)\a(b).
\end{align*}
From this we see that
\[
\|\a \circ \pi(a) - \pi(a) \circ \a\| = \|\a(a) - a\|,
\]
so that if we set $L(a) = \|\a(a)-a\|$ then $L$ will be a continuous
strongly Leibniz seminorm.  We can view this in another way.  View $A$
as a bimodule over $A$ by
\[
a \cdot b \cdot c = ab\a(c),
\]
and set $d(a) = \a(a) - a$.  It is easily checked that $d$ is a
(continuous) derivation, and so from Proposition~{\ref{prop2.1}} we
see again that $L$ is strongly Leibniz. (This does not require
that $\a$ be isometric.)
\end{example}

\begin{example}
\label{exam2.5}
Now let $G$ be a group, and let $\a$ be an action of $G$ on $A$, that
is, a homomorphism from $G$ into $\Aut(A)$.  Let $\ell$ be a
length-function on $G$.  For each $x \in G$ with $x \ne e_G$ the map
$a \mapsto \|\a_x(a)-a\|/\ell(x)$ is a continuous strongly Leibniz
seminorm.  Let $L$ be the supremum over $G$ of all of these seminorms,
so that
\[
L(a) = \sup\{\|\a_x(a)-a\|/\ell(x): x \ne e_G\}.
\]
By Proposition~\ref{prop1.2} we see that $L$ is a
lower-semicontinuous strongly-Leibniz seminorm.  Of course $L$ may not
be semifinite.  But if $G$ is a locally compact group, if $A$ is
complete, so a Banach algebra, if $\a$ is a strongly continuous action
by isometric automorphisms of $A$, and if $\ell$ is a continuous
length-function, then the discussion before theorem~$2.2$ of
\cite{R19} shows that $L$ is semifinite.  The discussion there is
stated just for $C^*$-algebras, but it applies without change to
Banach algebras.
\end{example}

\begin{example}
\label{exam2.6}
Suppose now that $G$ is a connected Lie group, and that $\a$ is a
strongly continuous action of $G$ on $A$ by isometric automorphisms.
Let $\fG$ denote the Lie algebra of $G$, and let $A^{\i}$ denote the
dense subalgebra of smooth elements of $A$ for the action $\a$.  We
let $\a$ also denote the corresponding infinitesimal action of $\fG$
on $A^{\i}$, defined by
\[
\a_X(a) = \left. \frac {d}{dt}\right|_{t=0} \a_{\exp(tX)}(a)
\]
for $X \in \fG$ and $a \in A^{\i}$.  The argument in the proof of
lemma~$3.1$ of \cite{R19} works here, and shows that
\[
\|\a_X(a)\| = \sup\{\|\a_{\exp(tX)}(a) - a\|/|t|: t \ne 0\}.
\]
It follows from Proposition~\ref{prop1.2} that the map $a \mapsto
\|\a_X(a)\|$ is a finite lower-semicontinuous strongly-Leibniz
seminorm on $A^{\i}$.  Suppose further that we are given a norm on
$\fG$, and that we set
\[
L(a) = \sup\{\|\a_X(a)\|: \|X\| \le 1\}.
\]
It follows again from Proposition~\ref{prop1.2} that $L$ is a
lower-semicontinuous strongly-Leibniz seminorm on $A^{\i}$, which is
easily seen to be finite, but which may well not be norm-continuous.
\end{example}

\begin{example}
\label{exam2.7}
Suppose now that $G$ is a Lie group and that $(U,\cH)$ is a strongly
continuous representation of $G$ on a Hilbert space $\cH$.  As
discussed in section~$3$ of \cite{R19} we can define an action, $\a$,
of $G$ on $\cB(\cH)$ by $\a_x(T) = U_xTU_x^*$, and we can let $B$ be
the largest subalgebra of $\cB(\cH)$ on which this action is strongly
continuous.  We can then apply the discussion of the previous example
to obtain a seminorm $L$ on $B^{\i}$.  If $A$ is a unital
$*$-subalgebra of $B^{\i}$ (which need not be carried into itself by
$\a$), then according to Proposition~\ref{prop1.4} the restriction of
$L$ to $A$ is a lower-semicontinuous strongly-Leibniz $*$-seminorm
which is clearly finite.
\end{example}

\begin{example}
\label{exam2.8}
Let us consider the above situation for the special case in which $\fG
= \bR$.  Then $U$ is generated by a self-adjoint (often unbounded)
operator, $D$, on $\cH$, that is, $U_t = e^{itD}$ for all $t \in \bR$.
Then it follows easily that for $T \in B^{\i}$
\[
L(T) = \|[D,T]\|,
\]
and in particular that the commutator $[D,T]$ is a bounded operator.
All of this will then be true for any $T \in A \subset B^{\i}$.  This
applies in particular to the ``Dirac'' operators on which Connes
\cite{C2,GVF} bases his approach to metric non-commutative
differential geometry.
\end{example}


\section{Closed seminorms}
\label{sec3}

We adapt here some definitions from section~$4$ of \cite{R5}.  Let $A$
be a normed unital algebra, and let ${\bar A}$ denote its completion.
Let $L$ be a seminorm on $A$ (value $+\i$ allowed) and let
\[
\cL_1 = \{a \in A: L(a) \le 1\}.
\]
Let ${\bar \cL}_1$ be the closure of $\cL_1$ in ${\bar A}$, and let
${\bar L}$ denote the corresponding ``Minkowski functional'' on ${\bar
A}$, defined by setting, for $c \in {\bar A}$,
\[
{\bar L}(c) = \inf\{r \in \bR^+: c \in r{\bar \cL}_1\}.
\]
The value $+\i$ must be allowed.  Then ${\bar L}$ is a seminorm on
${\bar A}$, and the proof of proposition~$4.4$ of \cite{R5} tells us
that if $L$ is lower semicontinuous, then ${\bar L}$ is an extension
of $L$.  We call ${\bar L}$ the {\em closure} of $L$.  We see
that the set $\{c \in {\bar A}: {\bar L}(c) \le 1\}$ is closed in
${\bar A}$.  We say that the original seminorm $L$ on $A$ is {\em
closed} if $\cL_1$ is closed in ${\bar A}$, or, equivalently, is
complete for the norm on $A$.  Clearly if $L$ is closed, then it is
lower semicontinuous.  If $L$ is closed and is not defined on all of
${\bar A}$, then ${\bar L}$ is obtained simply by giving it value
$+\i$ on all the elements of ${\bar A}$ that are not in $A$.  It is
clear that if $L$ is semifinite then so is ${\bar L}$. We recall that
a unital subalgebra $B$ of a unital algebra $A$ is said to be spectrally
stable in $A$ if for any $b \in B$ the spectrum of $b$ as an element
of $B$ is the same as its spectrum as an element of $A$, or equivalently,
that any $b$ that is invertible in $A$ is invertible in $B$.

\begin{proposition}
\label{prop3.1}
Let $L$ be a Leibniz seminorm on a normed unital
algebra $A$.  Then ${\bar L}$ is Leibniz. Set
\[
{\bar A}^f = \{c \in {\bar A}: {\bar L}(c) < \i\}.
\]
If $L(1) < \infty$, then ${\bar A}^f$ is a unital spectrally-stable subalgebra of the
norm closure of $\bar A^f$ in ${\bar
A}$. If $A$ is defined over $\bC$, then $\bar A^f$ is stable under the 
holomorphic-function calculus of its closure.
\end{proposition}

\begin{proof}
Let $c,d \in {\bar A}$.  If ${\bar L}(c) = \i$ or ${\bar L}(d) = \i$
there is nothing to show for the Leibniz condition.  Otherwise, we can find sequences $\{a_n\}$
and $\{b_n\}$ in $A$ such that $\{a_n\}$ converge to $c$ while
$\{L(a_n)\}$ converges to $\bar L(c)$ and $L(a_n) \le \bar L(c)$ for all $n$,
and similarly for $\{b_n\}$ and $d$.  Then $a_nb_n$ converges to $cd$
and 
\[
L(a_nb_n) \le L(a_n)\|b_n\| + \|a_n\|L(b_n) \le \bar L(c)\|b_n\| +
\|a_n\| \bar L(d),
\]
and the right-hand side converges to $\bar L(c)\|d\| + \|c\| \bar L(d)$.
Thus $\bar L$ is Leibniz. 

If $L(1) < \infty$ so that $\bar A^f$ is a unital subalgebra of $\bar A$, 
it follows from proposition 3.12 of \cite{BlC}
(or proposition 1.7 and theorem 1.17 of \cite{Sch}
or lemma 1.6.1 of \cite{Cu}) that $\bar A^f$ is
spectrally stable in its closure in $\bar A$, and in fact is stable under the
holomorphic-function calculus there. We sketch the proof in our
simpler setting. Define a new norm, $M$, on $\bar A^f$ by
\[
M(c) = \|c\| + \bar L(c).
\]
Then, as mentioned after definition~$4.5$ of \cite{R5}, $\bar A^f$ will be
complete for the norm $M$ because $\bar L$ is closed.  
(See the proof of proposition~$1.6.2$ of
\cite{Wvr2}.)  Because $\bar L$ is Leibniz, $M$ is easily seen to be
an algebra norm, so that $\bar A^f$ becomes a Banach algebra.
Let $c \in \bar A^f$. From the Leibniz rule we find that $\bar L(c^n) \leq n\|c\|^{n-1} \bar L(c)$,
so that
\[
M(c^n) \leq \|c\|^n + n\|c\|^{n-1}\bar L(c).
\]
From this it follows that if $\|c\| < 1$ then the series 
$\sum_{n = 0}^\infty c^n$ converges for $M$ to an element of $\bar A^f$. 
Thus $1-c$ is invertible in $\bar A^f$. It follows that if instead $\|1-c\| < 1$ then
$c$ is invertible in $\bar A^f$. From this it is then easily seen (e.g. lemma 3.38 of \cite{GVF})
that if $c \in \bar A^f$ and if $c$ is invertible in the norm-closure of $\bar A^f$ in $\bar A$, 
then $c$ is invertible in $\bar A^f$. Consequently $\bar A^f$ is spectrally stable in its closure
in $\bar A$.

Assume now that $A$ is defined over $\bC$.
For the definition and properties of the holomorphic-function (or
``symbolic'') calculus see \cite{KdR, Rdn}. It is well-known that a
dense subalgebra that is spectrally stable and is a Banach algebra
for a norm stronger that the norm of the bigger algebra, is stable
under the holomorphic-function calculus. (See the comments after
definition 3.25 of \cite{GVF}.) We briefly recall the reason, for our context.
For notational simplicity we assume that $\bar A^f$ is dense in $\bar A$. 
Let $c \in {\bar A}^f$, and let $f$ be a $\bC$-valued
function defined and holomorphic on an open neighborhood $\cO$ of the
spectrum $\s_{\bar A}(c)$.  Let $\g$ be a union of a finite number of
curves in $\cO$ that surrounds $\s_{\bar A}(c)$ in the usual way such that
the Cauchy integral formula using $\g$ gives $f$ on a neighborhood of
$\s_{\bar A}(c)$.  Since ${\bar A}^f$ is spectrally stable in ${\bar A}$, 
the function $z \mapsto (z-c)^{-1}$,
well-defined on $\g$, has values in ${\bar A}^f$.
Since $\bar A^f$ is a Banach algebra for $M$, this function is
continuous for $M$, and the integral
\[
f(c) = \frac {1}{2\pi i} \int_{\g} f(z)(z-c)^{-1}dz
\]
is well defined in ${\bar A}^f$.  Since the homomorphism from ${\bar
A}^f$ with norm $M$ to $\bar A$ with its original norm is clearly
continuous, the image of $f(c)$ in $\bar A$ will be expressed by the
same integral but now interpreted in $\bar A$.  But $f(c) \in {\bar
A}^f$.  So the above integral, but interpreted in $\bar A$, gives an
element of ${\bar A}^f$ as was to be shown.
\end{proof}

For the use of the holomorphic-function calculus when dealing with
algebras over $\bR$ see proposition 2.4 of \cite{R17}.

One reason that the property of being closed under the
holomorphic-function calculus is important is that it implies that
${\bar A}^f$ and its closure, say $B$, in ${\bar A}$ have essentially the
same finitely-generated projective modules (``vector bundles'') in the
sense that any such right module $V$ for $B$ is of the form $V = W
\otimes_{{\bar A}^f} B$ for such a right module $W$ over ${\bar A}^f$,
unique up to isomorphism. This is crucial to
\cite{R17}, and to our proposed discussion of projective modules and
quantum Gromov--Hausdorff distance for non-commutative $C^*$-algebras. 
The inclusion of ${\bar A}^f$ into $B$ also
gives an isomorphism of their $K$-groups.  (See appendix~IIIC of
\cite{C2} and theorem~$3.44$ of \cite{GVF}.)  

\begin{proposition}
\label{prop3.2}
Let $L$ be a strongly-Leibniz seminorm on a
normed algebra $A$. Assume that $A^f$ is dense and spectrally stable in $\bar A$.
Then the closure, ${\bar L}$, of $L$ is strongly Leibniz.
\end{proposition}

\begin{proof}
It is clear that ${\bar L}(1) = 0$.  From Proposition~\ref{prop3.1} we
know that ${\bar L}$ is Leibniz. Thus we only need to verify the
condition on inverses. Suppose now that $c \in {\bar A}$
and that $c$ is invertible in ${\bar A}$. If ${\bar L}(c) = \i$ there
is nothing to show, so assume that $c \in {\bar A}^f$.  Then there is
a sequence $\{a_n\}$ in $A$ that converges to $c$ while $\{L(a_n)\}$
converges to ${\bar L}(c)$ with $L(a_n) \le {\bar L}(c)$ for all $n$ (so $a_n \in A^f$).
Since $c$ is invertible in ${\bar A}$, and the set of invertible
elements of a unital Banach algebra is open, the elements $a_n$ are
eventually invertible in ${\bar A}$.  
Since $A^f$ is assumed to be 
spectrally stable in $\bar A$ 
the elements $a_n$ are eventually
invertible in $A^f$.  Thus we can adjust the sequence $\{a_n\}$ so that
each element is invertible in $A^f$.  Then the sequence $\{a_n^{-1}\}$
converges to $c^{-1}$, while for each $n$
\[
L(a_n^{-1}) \le \|a_n^{-1}\|^2L(a_n) \le \|a_n^{-1}\|^2{\bar L}(c).
\]
It follows easily that ${\bar L}(c^{-1}) \le \|c^{-1}\|^2{\bar L}(c)$.
Thus ${\bar L}$ is strongly Leibniz.
\end{proof}


\section{$C^*$-metrics}
\label{sec4}

Up to this point we have ignored the crucial analytic property of the
seminorms that define quantum metric spaces, i.e., $\Lip$-norms.  We
recall this property here, for our special context of unital
$C^*$-algebras.  Let $A$ be a unital $*$-algebra equipped with a
$C^*$-norm (but not assumed to be complete).  
Let $L$ be a seminorm on
$A$ such that $L(1) = 0$.  Define a metric, $\rho_L$, on the state
space, $S(A)$, of $A$ by
\[
\rho_L(\mu,\nu) = \sup\{|\mu(a)-\nu(a)|: a = a^* \mbox{ and } L(a) \le
1\}.
\]
(Without further hypotheses $\rho_L$ might take the value $+\i$.)  We
will say that $L$ is a $\Lip$-norm if the topology on $S(A)$ from
$\rho_L$ coincides with the weak-$*$ topology on $S(A)$.  In our
definition of $\Lip$-norms in definition~$2.1$ of \cite{R6} we, in
effect, assumed that our seminorms $L$ were defined only on the
self-adjoint part of $A$, but still defined $\rho_L$ as above.  The
comments before definition~$2.1$ of \cite{R6} show that if $L$ is a
$*$-seminorm then $\rho_L$ would not change if the condition ``$a =
a^*$'' above were omitted.

We now come to the definition that seems to be dictated by our
investigation of vector bundles and Gromov--Hausdorff distance, both
for ordinary metric spaces \cite{R17} and for quantum ones. It
should be viewed as tentative, since future experience may require
additional hypotheses.

\begin{definition}
\label{def4.1}
Let $A$ be a unital $C^*$-normed algebra and let $L$ be a seminorm on
$A$ (possibly taking value $+\i$).  We will say that $L$ is a {\em
$C^*$-metric} on $A$ if
\begin{itemize}
\item[{\em a)}] $L$ is a lower-semicontinuous strongly-Leibniz
$*$-seminorm,
\item[{\em b)}] $L$ (restricted to $A^{sa}$) is a $\Lip$-norm,
\item[{\em c)}] $A^f$ is spectrally stable in the completion, ${\bar
A}$, of $A$.
\end{itemize}
By a {\em compact $C^*$-metric space} we mean a pair $(A,L)$
consisting of a unital $C^*$-normed algebra $A$ and a $C^*$-metric $L$
on $A$.
\end{definition}

In using the word ``space'' above, we should logically be referring to
objects in the dual to the category of unital $C^*$-algebras.  But we
will not make this distinction explicit during our discussions in this
paper.

We need condition c) in Definition \ref{def4.1} so that we can apply
Proposition \ref{prop3.2} to conclude that the closure of a $C^*$-metric 
is strongly Leibniz and itself satisfies condition c). Hanfeng Li has
pointed out to me that the subalgebra of polynomials in the algebra of 
continuous functions on the unit interval with the usual Lipschitz 
seminorm shows that condition c) is independent of conditions a) and b).

At this time it is not clear to me how best to define $C^*$-metric
spaces that are locally compact but not compact, though some
substantial indications can be gleaned from the results in \cite{Ltr}.

Recall \cite{BlC} that a $*$-subalgebra $B$ of a $C^*$-algebra $A$ is said
to be stable under the $C^2$-function calculus for self-adjoint elements
if for any $b \in B$ with $b^* = b$ and any twice continuously
differentiable function $f$ on $\bR$, the element $f(b)$ of
$A$, defined by the continuous-function calculus on self-adjoint
elements of $A$, is in fact again in $B$.

\begin{proposition}
\label{prop4.2}
Every $C^*$-metric on a unital $C^*$-normed algebra is semifinite.
Let $L$ be a $C^*$-metric on a unital $C^*$-normed algebra $A$, and
let $\bar L$ be its closure on the completion $\bar A$ of $A$ (so
$\bar L$ is an extension of $L$). Then $\bar L$ is a $C^*$-metric. Let
$\bar A^f$ be defined as earlier (so now $\bar A^f$ is dense in
$\bar A$). Then $\bar A^f$ is stable both under the holomorphic-function
calculus of $\bar A$ and the $C^2$-calculus on self-adjoint
elements of $\bar A$. 
\end{proposition}

\begin{proof}
Let $L$ be a $C^*$-metric on a unital $C^*$-normed algebra $A$, and
let $A^f$ be defined as above.  Suppose that $A^f$ is not dense in
$A$.  Then it is easily seen that there is an $a \in
A$ with $a^* = a$ that is not in the closure of $A^f$.  By the
Hahn--Banach theorem there is a linear functional of norm~$1$ on the
self-adjoint part of $A$ that has value $0$ on all of the self-adjoint
part of $A^f$.  From lemma~$2.1$ of \cite{R5} it then follows that
there are two distinct states of $A$ which agree on $A^f$.  Then the
distance between these two states for the metric $\rho_L$ determined
by $L$ is $0$, which contradicts the requirement that the topology on
$S(A)$ determined by $\rho_L$ coincides with the weak-$*$ topology.

The fact that $\bar L$ is a $C^*$-metric is seen as follows. By
definition, $\bar L$ is closed so lower semi-continuous. As remarked
above, $\bar L$ is strongly Leibniz by condition c) and
Proposition \ref{prop3.2}. The closure of a Lip-norm is again a
Lip-norm, giving the same metric on the state-space, as seen
in proposition 4.4 of \cite{R5}. That $\bar L$ satisfies condition c)
follows from Proposition \ref{prop3.1}.

The fact that $\bar A^f$ is stable under the holomorphic-function
calculus of $\bar A$ follows immediately from the semifiniteness of
$\bar L$ and Proposition \ref{prop3.1}. The fact that $\bar A^f$ is
stable for the $C^2$-function calculus on self-adjoint elements
of $\bar A$ follows quickly from proposition 6.4 of \cite{BlC}, which
actually gives a slightly stronger fact.
\end{proof}

The condition that $L$ be a $\Lip$-norm is often a difficult one to
verify for various specific examples.  But most of the $\Lip$-norms
that have been constructed on $C^*$-algebras so far are in fact
$C^*$-metrics.  We explain this now for several of the classes of
examples described in sections~$2$ and $3$ of \cite{R19}.

\begin{example}
\label{exam4.3}
Let $A$ be a unital $C^*$-algebra, let $G$ be a compact group, and let
$\a$ be an action of $G$ on $A$ that is ergodic in the sense that if
an $a \in A$ satisfies $\a_x(a) = a$ for all $x \in G$ then $a \in \bC
1_A$.  Let $\ell$ be a continuous length function on $G$, and define a
seminorm $L$ on $A$, as in Example~\ref{exam2.5}, by
\[
L(a) = \sup\{\|\a_x(a)-a\|/\ell(x): x \notin e_G\}.
\]
It is shown in \cite{R19} that $L$ is a $\Lip$-norm.  But we saw in
Example~\ref{exam2.5} that $L$ is lower semicontinuous and strongly
Leibniz. Since $L$ is defined on all of $A$, the spectral stability
of $\bar A^f$ in $A$ follows from Proposition \ref{prop3.1}.
\end{example}

The next several examples involve ``Dirac'' operators in various
settings.

\begin{example}
\label{exam4.4}
This class of examples is the main class discussed in Connes's first
paper \cite{Cn7} on metric aspects of non-commutative geometry.  It is
discussed briefly as example~$3.6$ of \cite{R19}.  Let $G$ be a
discrete group and let $A = C_r^*(G)$ be its reduced group
$C^*$-algebra acting on $\ell^2(G)$.  Let $\ell$ be a length function
on $G$.  As Dirac operator take the operator $D = M_{\ell}$ of
pointwise multiplication by $\ell$ on $\ell^2(G)$.  The one-parameter
unitary group generated by $D$ simply sends $t$ to the operator of
pointwise multiplication by the function $e^{it\ell}$.  We are then in
the context of Examples~\ref{exam2.7} and \ref{exam2.8}.  It is easily
seen that the dense subalgebra $C_c(G)$ of functions of finite support
is in the smooth algebra $B^{\i}$ for the action of $G$ on
$\cL(\ell^2(G))$.  As in Example~\ref{exam2.8} we thus obtain a
lower-semicontinuous strongly-Leibniz semifinite $*$-seminorm on $A$,
which for any $f \in C_c(G)$ is given by
\[
L(f) = \|[D,f]\|.
\]
From Proposition~\ref{prop3.1} it follows that $A^f$ (for the closure
of $L$) is stable for the
holomorphic-function calculus.  But for stupid length functions $L$
can fail to be a $\Lip$-norm, and it is not easy to see when it is a
$\Lip$-norm, and thus a $C^*$-metric.  In \cite{R18}, by means of a
long and interesting argument, it is shown that $L$ is a $\Lip$-norm,
and thus a $C^*$-metric,
for $G = \bZ^d$ (and even for the twisted group algebra
$C^*(\bZ^d,\g)$ where $\g$ is a bicharacter on $\bZ^d$) when $\ell$ is
either a word-length function or the restriction to $\bZ^d$ of a norm
on $\bR^d$.  In \cite{OzR} it is shown, by techniques entirely
different from those used for the case of $\bZ^d$, that if $G$ is a
hyperbolic group and $\ell$ is a word-length function on $G$ then $L$
is a $\Lip$-norm, and thus a $C^*$-metric.  For other classes of
infinite discrete groups, e.g., nilpotent ones, it remains a mystery
as to whether $L$ is a $\Lip$-norm if $\ell$ is a word-length
function.  Some related examples can be found in \cite{AnC}.
\end{example}

\begin{example}
\label{exam4.5}
Let $\a$ be an action of the $d$-dimensional torus $\bT^d$, $d \ge 2$,
on a unital $C^*$-algebra $A$.  In \cite{R9} it is shown that for any
skew-symmetric real $d \x d$ matrix $\th$ one can deform the product
on $A$ to get a new $C^*$-algebra, $A_{\th}$.  Connes and Landi
\cite{CnL} show that when $M$ is a compact spin Riemannian manifold
and $\a$ is a smooth action of $\bT^d$ on $M$, so on $A = C(M)$,
leaving the Riemannian metric invariant, and lifting to the spin
bundle, then there is a natural Dirac operator for the (usually
non-commutative) deformed algebra $A_{\th}$.  As in
Examples~\ref{exam2.8} and \ref{exam4.3}, this Dirac operator
determines a $*$-seminorm, $L$, on $A_{\th}$ which is lower
semicontinuous, strongly Leibniz, and semifinite. 
Hanfeng Li \cite{Lih1} showed that $L$ is a $\Lip^*$-norm.
Thus $L$ is a $C^*$-metric.
\end{example}


\section{Quotient seminorms and proximity}
\label{sec5}

We now try to modify the definition of quantum Gromov--Hausdorff
distance so as to use the above definition of $C^*$-metrics.  This
involves quotient seminorms, so we begin by exploring them.  There are
at least three difficulties that confront us, namely that the quotient
of a Leibniz seminorm may not be Leibniz, that the quotient of a
strongly Leibniz seminorm, even if it is Leibniz, may not be strongly
Leibniz, and that reasonable $*$-seminorms can agree on self-adjoint
elements but still be distinct.  We begin by considering the first
difficulty.

Let $L$ be a Leibniz seminorm on a unital normed algebra $C$, and let
$\pi: C \twoheadrightarrow  A$ be a unital homomorphism from $C$ onto
a unital normed algebra $A$.  Let ${\tilde L}^A$ be the quotient
seminorm on $A$, defined by
\[
{\tilde L}^A(a) = \inf\{L(c): c \in C \mbox{ and } \pi(c) = a\}.
\]
It is known \cite{BlC} that ${\tilde L}^A$ need not be Leibniz.  (See
also lemma~$4.3$ of \cite{Lih3} and the comments just before it.)  But
the situation can be partly rescued by the following definition.

\begin{definition}
\label{def5.1}
Let $C$, $A$, $\pi$ and $L$ be as above, and assume that $\pi$ is norm
non-increasing.  We say that $L$ is {\em $\pi$-compatible} if for
every $a \in A$ and every $\e > 0$ there is a $c \in C$ such that
$\pi(c) = a$ and simultaneously
\[
L(c) \le {\tilde L}^A(a) + \e \qquad \mbox{ and } \qquad \|c\| \le
\|a\| + \e.
\]
\end{definition}

\begin{proposition}
\label{prop5.2}
Let $C$, $A$, $\pi$ and $L$ be as above.  If $L$ is $\pi$-compatible
then the norm on $A$ coincides with the quotient norm from $C$, and
${\tilde L}^A$ is Leibniz.
\end{proposition}

\begin{proof}
The statement about the norms is easily verified.  Suppose now that
$a,b \in A$ and $\e > 0$ are given.  Since $L$ is $\pi$-compatible, we
can find $c,d \in C$ such that $\pi(c) = a$ and $\pi(d) = b$ and the
conditions of Definition~\ref{def5.1} are satisfied.  Then $\pi(cd) =
\pi(ab)$, and so
\begin{align*}
{\tilde L}^A(ab) &\le L(cd) \le L(c)\|d\| + \|c\|L(d) \\
&\le ({\tilde L}^A(a) + \e)(\|b\| + \e) + (\|a\| + \e)({\tilde L}^A(b)
+ \e).
\end{align*}
Since $\e$ is arbitrary, we see that ${\tilde L}^A$ is Leibniz.
\end{proof}

However, I do not know of a useful way to partly rescue the difficulty
that if $L$ is strongly Leibniz and ${\tilde L}^A$ is Leibniz there
seems to be no reason that ${\tilde L}^A$ need be strongly Leibniz
(though I do not have an example showing this difficulty).

We now consider the third difficulty.  It is quite instructive to
first consider ordinary metric spaces.  For this purpose
$\pi$-compatibility is useful.

\begin{proposition}
\label{prop5.3}
Let $(Z,\rho)$ be a compact metric space, and let $C = C(Z)$ be its
$C^*$-algebra of continuous complex-valued functions.  Let $X$ be a
closed subset of $Z$, let $A = C(X)$, and let $\pi: C \to A$ be the
usual restriction homomorphism.  Then the Leibniz seminorm $L^{\rho}$
for $\rho$ is $\pi$-compatible.
\end{proposition}

\begin{proof}
Let $f \in A$.  Let $Q$ be the radial retraction of $\bC$ onto its
ball of radius $\|f\|_{\i}$ centered at $0$.  It is easily seen that
the Lipschitz constant of $Q$ is $1$.  Then for any $h \in C$ with
$\pi(h) = f$ we can set $g = Q \circ h$ and we will have $\pi(g) = f$
and $L(g) \
\le L(h)$ while $\|g\| = \|f\|$.  This quickly gives the desired
result.
\end{proof}

We remark that the above argument does not work for matrix-valued
functions, as employed in \cite{R17}, since the radial retraction no
longer has Lipschitz constant $1$ \cite{R20}.

While Proposition~\ref{prop5.3} appears favorable, the difficulty is
that the quotient of $L^{\rho}$ on $A$ need not agree with the
Lipschitz seminorm from the metric $\rho_X$ on $X$ coming from
restricting $\rho$:

\begin{example}
\label{exam5.4}
(See \cite{Wvr2,R20}.)  Let $(X,\rho_X)$ be the metric space 
containing exactly $3$
points, at distance $2$ from each other.  We can ask what the
Gromov--Hausdorff distance is from $(X,\rho_X)$ to a metric space
consisting of one point, say $p$.  It is easily seen that the answer
is $1$, with the metric $\rho$ on $Z = X \cup \{p\}$ that extends
$\rho_X$ giving $p$ distance $1$ to each point of $X$.  Now let $f$ be
the function on $X$ which sends the three points of $X$ to the three
different cube roots of $1$ in $\bC$.  It is not difficult to see that
the extension of $f$ to $Z$ that has the smallest Lipschitz norm is
the extension $g$ that sends $p$ to $0$.  But $L^{\rho}(g)$ is easily
seen to be substantially larger than $L^{\rho_X}(f)$.  As remarked in
\cite{R17,R20}, this is possible because the metric on $Z$ is somewhat
hyperbolic.
\end{example}

On the other hand, for any compact metric space $(Z,\rho)$, any closed
subset $X$ of $Z$, and for any $f \in C_{\bR}(X)$, there is a $g \in
C_{\bR}(Z)$ with $g|_X = f$, $\|g\| = \|f\|$ and $L^{\rho}(g) =
L^{\rho_X}(f)$  \cite{R20}.  This shows in particular that here
$L^{\rho_X}$ does coincide with the quotient seminorm from $L^{\rho}$.
It also means that for the situation of Example~\ref{exam5.4} we have
two Leibniz seminorms on $C(X)$ which agree on real-valued functions
but are nevertheless distinct. (For a related phenomenon see \cite{PlT}.)  
From the comments at the end of the
first paragraph of Section~\ref{sec4} we see that these two 
seminorms will give the
same metrics on the set of probability measures on $X$, and in
particular the same metrics on $X$.

We now turn our attention to Gromov--Hausdorff distance.  Let
$(A,L_A)$ and $(B,L_B)$ be $C^*$-metric spaces.  The evident way to
adapt the definition of quantum Gromov--Hausdorff distance given in
definition~$4.2$ of \cite{R6} is to require that the seminorms $L$
considered on $A \oplus B$ be $C^*$-metrics.  Example~\ref{exam5.4}
shows that we can not require the quotient of $L$ on $A$ to agree with
$L_A$, except on self-adjoint elements (though for the main class of
examples considered in later sections they will agree, so those
examples are better-behaved than Example~\ref{exam5.4}).  Then we do
not know whether the quotient is Leibniz.  We could impose
$\pi$-compatibility to ensure this, but then we still may not have the
strong Leibniz property, so it is not clear that it is useful to
impose this.

Perhaps as our topic develops in the future it will become clearer
what are the best conditions to impose.  Anyway, guided by the above
observations, we set, parallel to notation~$4.1$ of \cite{R6}:

\begin{notation}
\label{note5.5}
Let $(A,L_A)$ and $(B,L_B)$ be compact $C^*$-metric spaces.  We let
$\cM_C(L_A,L_B)$ denote the collection of all $C^*$-metrics, $L$, on
$A \oplus B$ such that the quotient of $L$ on $A$ agrees with $L_A$ on
self-adjoint elements of $A$, and similarly for the quotient of $L$ on
$B$.
\end{notation}

We want to modify the definition of quantum Gromov--Hausdorff
distance, $\dist_q$, given in definition~$4.2$ of \cite{R6} by
requiring that the seminorms involved there are in $\cM_C(L_A,L_B)$.
But I am not able to show that the resulting notion satisfies the
triangle inequality.  When one tries to imitate the proof of the
triangle inequality for $\dist_q$ given in theorem~$4.3$ of \cite{R6},
one of the main obstacles is in showing that the $\Lip$-norm $L_{AC}$
of lemma~$4.6$, which is defined as a quotient seminorm, is a
$C^*$-metric.  I would not be surprised if the triangle inequality
fails.  So the term ``distance'' should not be used.  I will use
instead the term ``proximity''.  Thus:

\begin{definition}
\label{def5.6}
Let $(A,L_A)$ and $(B,L_B)$ be compact $C^*$-metric spaces.  We define their
proximity by 
\[
\prox(A,B) = \inf\{\dist_H^{\rho_L}(S(A),S(B)): L \in
\cM_C(L_A,L_B)\}.
\]
\end{definition}

This definition makes sense in the following way.  Both $S(A)$ and
$S(B)$ are closed subsets of $S(A \oplus B)$. Much as at the beginning 
of Section \ref{sec4}, $\rho_L$ is a metric on $S(A\oplus B)$, and
$\dist_H^{\rho_L}$  is ordinary Hausdorff distance with respect to $\rho_L$.
We note that the hypotheses in the
definition of $\cM_C(L_A,L_B)$ are such that proposition~$3.1$ of
\cite{R6} applies, so that for any $L \in \cM_C(L_A,L_B)$ the
restrictions of $\rho_L$ to $S(A)$ and $S(B)$ coincide with
$\rho_{L_A}$ and $\rho_{L_B}$.  Put another way, when we associate to
each $L \in \cM_C(L_A,L_B)$ its restriction to the self-adjoint part
of $A \oplus B$ we obtain a map from $\cM_C(L_A,L_B)$ to
$\cM(L_A^s,L_B^s)$, where $L_A^s$ denotes the restriction of $L_A$ to
the self-adjoint part of $A$, and similarly for $L_B^s$.  This map
need not be either injective or surjective.

It is clear that
\[
\dist_q(A,B) \le \prox(A,B),
\]
since $\prox(A,B)$ is an infimum over a subset of the seminorms used
to define $\dist_q(A,B)$.  Thus if we have a sequence $(B^n,L_{B^n})$
of $C^*$-metric spaces for which the sequence $\prox(A,B^n)$ converges
to $0$, then it follows that $(B^n,L_{B^n})$ converges to
$(A,L_A)$ for quantum Gromov--Hausdorff distance.  For this reason the
absence of the triangle inequality will not be too serious a problem.
The advantage of $\prox$, as mentioned earlier, is that the use of
seminorms $L$ on $A \oplus B$ that are $C^*$-metrics permits one to
try to generalize to $C^*$-metric spaces the results about vector
bundles obtained in \cite{R17} for ordinary metric spaces.  (We plan
to discuss this in a future paper.)


\section{Bimodule-bridges}
\label{sec6}

In the development of quantum Gromov--Hausdorff distance given in
\cite{R6} and used in \cite{R7}, a very convenient method for
constructing suitable seminorms $L$ on $A \oplus B$ involved suitable
continuous seminorms $N$ on $A \oplus B$ that we called ``bridges'',
with $L$ then defined as
\[
L(a,b) = L_A(a) \vee L_B(b) \vee N(a,b).
\]
Within the context of the present paper it is natural to require that $N$ satisfy
a suitable Leibniz condition. There is an evident condition to consider, coming from
viewing $N$ as a seminorm on the algebra $A \oplus B$. But it
seems more appropriate to require the stronger condition
\[
N((a,b)(a',b')) \leq N(a,b)\|b'\| \ + \ \|a\|N(a',b').
\]
Examples show that this condition can be interpreted as indicating that $N$ only provides
metric data between $A$ and $B$, and not within $A$ or within $B$.

We will find  it very useful to use bridges that come from normed
bimodules. Such bridges will satisfy the Leibniz condition stated
above. Let $A$ and $B$ be unital $C^*$-algebras, and let $\O$ be
an $A$-$B$-bimodule.  We say that $\O$ is a normed bimodule if it is
equipped with a norm that satisfies, much as in Section~\ref{sec2},
\[
\|a\o b\| \le \|a\|\|\o\|\|b\|
\]
for all $a \in A$, $b \in B$ and $\o \in \O$.  We assume that the
identity elements of $A$ and  $B$ both act as the identity operator on
$\O$.

\begin{definition}
\label{def6.2}
Let $(A,L_A)$ and $(B,L_B)$ be $C^*$-metric spaces.  By a {\em
bimodule bridge} for $(A,L_A)$ and $(B,L_B)$ we mean a normed
$A$-$B$-bimodule $\O$ together with a distinguished element $\o_0 \ne
0$ such that when we form the seminorm $N$ on $A \oplus B$ defined by
\[
N(a,b) = \|a\o_0 - \o_0b\|,
\]
it has the property that for any $a \in A$ with $a = a^*$ and any $\e
> 0$ there is a $b \in B$ with $b^* = b$ such that
\[
L_B(b) \vee N(a,b) \le L_A(a) + \e,
\]
and similarly for $A$ and $B$ interchanged.
\end{definition}

\begin{theorem}
\label{th6.3}
Let $(\O,\o_0)$ be a bimodule bridge for the $C^*$-metric spaces
$(A,L_A)$ and $(B,L_B)$, and let $N$ be defined as above in terms of
$(\O,\o_0)$.  Define $L$ on $A \oplus B$ by
\[
L(a,b) = L_A(a) \vee L_B(b)  \vee N(a,b) \vee N(a^*,b^*).
\]
Then $L \in \cM_C(L_A,L_B)$.
\end{theorem}

\begin{proof}
One can show directly that $N$ is strongly Leibniz, or view $\O$ as an
$(A \oplus B)$-bimodule in the evident way and apply
Proposition~\ref{prop2.1}.  Since $N$ is also continuous, it follows
from Proposition~\ref{prop1.2} that $L$ is lower semicontinuous and
strongly Leibniz.  Clearly $L$ is a $*$-seminorm.  Thus condition a)
of Definition~\ref{def4.1} is satisfied.

We now want to apply theorem~$5.2$ of \cite{R6} to show that $L$ is a
$\Lip^*$-norm.  We must thus show that $N \vee N^*$, restricted to the
self-adjoint part of $A \oplus B$ is a bridge as defined in
definition~$5.1$ of \cite{R6}.  From its bimodule source it is clear
that $N(1_A,1_B) = 0$, while $N(1_A,0) \ne 0$ since $\o_0 \ne 0$.
Since also $N$ is continuous, it follows that the first two conditions
of definition~$5.1$ are satisfied.  The main technical condition of
Definition~\ref{def6.2} directly implies that condition $3$ of
definition~$5.1$ of \cite{R6} is satisfied, so that $N \vee N^*$ is
indeed a bridge, and so $L$, restricted to self-adjoint elements, is a
$\Lip$-norm.  Thus $L$ is a $\Lip^*$-norm, and so condition b) of
Definition~\ref{def4.1} is satisfied.  

Because $N$ is clearly finite, $(A \oplus B)^f$, as defined for $L$, 
coincides with $A^f \oplus B^f$.  From the fact that $A^f$ and $B^f$ 
are by assumption spectrally stable in their completion it follows easily 
that $(A \oplus B)^f$ is spectrally stable in its completion.  
Thus $L$ satisfies condition c) of Definition~\ref{def4.1}, 
so that $L$ is a $C^*$-metric.

Suppose now that we are given $a \in A$ with $a = a^*$.  From the
formula for $L$ it is clear that $L(a,b) \ge L_A(a)$ for all $b \in
B$.  Let $\e > 0$ be given.  Then by Definition~\ref{def6.2} there is
a $b \in B$ with $b = b^*$ such that
\[
L_B(b) \vee N(a,b) \le L_A(a) + \e.
\]
Since $N$ and $N^*$ agree on self-adjoint elements, it follows that
$L(a,b) \le L_A(a) + \e$.  Since $\e$ is arbitrary, it follows that
the quotient of $L$ on $A$ applied to $a$ gives $L_A(a)$.  In the same
way the quotient of $L$ on $B$, restricted to self-adjoint elements,
gives $L_B$ on self-adjoint elements. Thus $L \in \cM_C(L_A,L_B)$.
\end{proof}

In the next sections we will see how to construct useful bimodule
bridges for ``matrix algebras converging to the sphere''.

Hanfeng Li has pointed out to me that  prox  is dominated by the ``nuclear 
Gromov-Hausdorff distance'' $\mathrm{dist}_{nu}$ that he defined in remark 5.5 of \cite{Lih3}
and studied further in section 5 of \cite{KrL}. He gives a proof of this in the appendix
of \cite{Lih4}. (He uses the term ``nuclear'' because this
distance has favorable properties for nuclear $C^*$-algebras.) We sketch here how this
works, so that it can be easily compared with what we have done above. 
The crux of Li's approach is that he restricts attention to bimodules of a
quite special kind. Specifically, for unital $C^*$-algebras $A$ and $B$ let $\cH(A,B)$
denote the collection of all triples $(D, \iota_A, \iota_B)$ consisting of a unital $C^*$-algebra
$D$ and injective (so isometric) 
unital homomorphisms $\iota_A$ and $\iota_B$ from $A$ and $B$
into $D$. We can then view $D$ as an $A$-$B$-bimodule in the evident way. For a $C^*$-metric $L_A$ on $A$ Li sets
\[
\cE(L_A) \ = \ \{a \in A^{sa}:\ L_A(a) \leq 1\},
\]
the $L_A$-unit-ball in $A^{sa}$. Then for any $(D, \iota_A, \iota_B) \in \cH(A, B)$
he considers
\[
\mathrm{dist}_H(\iota_A(\cE(L_A)), \iota_B(\cE(L_B))),
\]
 the ordinary Hausdorff distance in $D$ for the norm of $D$. Even though $\cE(L_A)$
 and $\cE(L_B)$ are unbounded, this distance is finite, for the following reason. Let
 $r_A$ be the radius of $(A, L_A)$, as defined in section 2 of \cite{R5}, so 
 that $\|\tilde a\| \ \tilde {} \ \leq \ r_A L_A(a)$ for any $a \in A^{sa}$, where \ $\tilde{}$ \ denotes
 image in the quotient $A^{sa}/\bR 1_A$, with the quotient norm. Then if
 $a \in \cE(L_A)$ so that $L_A(a) \leq 1$, it follows that $a = a' + t1_A$
 for some $t \in \bR$ and $a' \in A^{sa}$ with $\|a'\| \leq r_A$. Let $b = t1_B$,
 so that $b \in \cE(L_B)$. Then
 \[
 \|\iota_A(a) - \iota_B(b)\| = \|a'\| \leq r_A .
 \]
 Thus $\iota_A(\cE(L_A))$ is in the $r_A$-neighborhood of 
 $\iota_B(\cE(L_B))$.
By also interchanging the roles of $a$ and $b$ we see that
\[
\mathrm{dist}_H(\iota_A(\cE(L_A)), \iota_B(\cE(L_B)) \ \leq \ \max(r_A, r_B).
\] 
Then Li defines $\mathrm{dist}_{nu}(A, B)$ (or, more precisely, $\mathrm{dist}_{nu}(L_A, L_B)$) to be
\[
 \ \inf\{\mathrm{dist}_H(\iota_A(\cE(L_A)), \iota_B(\cE(L_B))):
(D, \iota_A, \iota_B) \in \cH(A, B)\} .
\]

Li shows as follows that $\mathrm{dist}_{nu}$ satisfies the triangle inequality. Suppose that a 
third compact $C^*$-metric space $(C, L_C)$ is given. Let 
$d_{AB} = \mathrm{dist}_{nu}(A,B)$, and similarly for $d_{BC}$ and $d_{AC}$. 
Given $\e > 0$ we can find $(D, \iota_A, \iota_B) \in \cH(A, B)$ and
 $(E, \rho_B, \rho_C) \in \cH(B, C)$ such that
 \[
 \mathrm{dist}_H(\iota_A(\cE(L_A)), \iota_B(\cE(L_B)) \ \leq \
 d_{AB} \ + \ \e ,
 \]
 and similarly for $d_{BC}$.
 Let $F = D*_B E$ be \emph{an} amalgamated product of $D$ and $E$ over $B$
 (using the inclusions $\iota_B$ and $\rho_B$). This means that there are unital injective
 homomorphisms $\s_D$ and $\s_E$ of $D$ and $E$ into $F$ such that
 $\s_D \circ \iota_B \ = \ \s_E \circ \rho_B$. (It is natural to cut down to the subalgebra
 generated by the images of $D$ and $E$ in $F$.) 
 
 Before continuing, we remark that it is easy to construct a universal 
 amalgamated free product, $A*_C B$, if one does
 not insist that the homomorphisms into it from $A$ and $B$ are injective. One takes
 the quotient of the universal (i.e. full) free product $A*B$ by the ideal generated by 
 the desired relations from $C$. See \cite{Lrn}. What is not as simple is to show
 that the evident homomorphisms of $A$ and $B$ into the universal 
 $A*_C B$ are injective.
 This was first shown by Blackadar in \cite{Blk}. In a comment added in proof in
 that paper, Blackadar says that John Phillips has shown him a preferable proof. 
 Blackadar has shown me this proof of John Phillips, and since it seems
 not to have appeared in print up to now, we sketch it here. Hanfeng Li
 has pointed out to me that a version of the
 argument in a substantially more complicated situation appears in the proof
 of proposition 2.2 of \cite{ADE}.
 
 To simplify notation we simply view $C$ as a unital subalgebra of each of $A$
 and $B$. The crux of the matter is to show that there are faithful (non-degenerate)
 representations of $A$ and $B$ on the same Hilbert space whose restrictions
 to $C$ are equal. We construct such representations as follows.
 
 \begin{enumerate}
 \item Let $(\pi_1, \cH_1)$ be a faithful representation of $A$. Form the restricted
 representation $(\pi_1 |_C, \cH_1)$ of $C$, and extend it to a representation
 $(\rho_1, \cH_1 \oplus \cK_1)$ of $B$. (This can be done by decomposing
 into cyclic representations and extending their states -- see lemma 2.1 of \cite{ADE} .)
 \item Notice that $\rho_1|_C$ carries $\cH_1$ into itself and so carries $\cK_1$
 into itself. Extend $(\rho_1|_C, \cK_1)$ to a representation $(\pi_2, \cK_1\oplus \cH_2)$
 of $A$.
 \item Extend $(\pi_2|_C, \cH_2)$ to a representation $(\rho_2, \cH_2 \oplus \cK_2)$
 of $B$.
 \item Continue this process through all the positive integers, and form
 $\cH = \bigoplus_1^\infty (\cH_j \oplus \cK_j)$. The $\pi_j$'s and $\rho_j$'s
 combine to give representations $\pi$ and $\rho$ of $A$ and $B$ on $\cH$
 which can be checked to agree on $C$. Since $\pi_1$ was chosen to be
 a faithful representation of $A$, so is $\pi$. Thus the homomorphism
 from $A$ into $A*_C B$ must be injective. The situation is symmetric for $A$
 and $B$, so the homomorphism
 from $B$ into $A*_C B$ must also be injective.
 \end{enumerate}

We return to demonstrating the triangle inequality for $\mathrm{dist}_{nu}$. 
Let $\tau_A = \s_D\circ \iota_A$ and $\tau_C = \s_E \circ \rho_C$.
 Then $(F, \tau_A, \tau_C) \in \cH(A,C)$. Furthermore, if $a \in \cE(L_A)$ then there
 is a $b \in \cE(L_B)$ such that $\|\iota_A(a) - \iota_B(b)\| \leq d_{AB} + \e$, so that
$ \|\tau_A(a) - \s_D(\iota_B(b))\| \leq d_{AB} + \e$. In the same way there exists
$c \in \cE(L_C)$ such that $\|\s_E(\rho_B(b)) - \tau_C(c))\| \leq d_{BC} + \e$.
But $ \s_D(\iota_B(b)) = \s_E(\rho_B(b))$, and so
\[
\|\tau_A(a) - \tau_C(c)\| \leq d_{AB} + d_{BC} + 2\e.
\]
In this way we find that
\[
\mathrm{dist}_{nu}(L_A, L_C) \ \leq \ \mathrm{dist}_{nu}(L_A, L_B) \ + \ \mathrm{dist}_{nu}(L_B, L_C) .
\]
Further favorable properties of $\mathrm{dist}_{nu}$ are presented in \cite{Lih3, KrL} that we
will not discuss here.

Given $(D, \iota_A, \iota_B) \in \cH(A, B)$, we can view $D$ as a normed $A$-$B$-bimodule
in the evident way, and as special element we can choose
$\omega_0 = 1_D$. The corresponding bounded seminorm $N_D$ on $A \oplus B$
is then simply defined by
\[
N_D(a,b) = \|\iota_A(a) - \iota_B(b)\| .
\] 
Given $C^*$-metrics $L_A$ and $L_B$ on $A$ and $B$, we can seek constants $\g$
such that $\g^{-1}N_D$ is a bimodule bridge for $L_A$ and $L_B$. Let 
$\d = \mathrm{dist}_H(\iota_A(\cE(L_A)), \iota_B(\cE(L_B)))$. Given any $\e > 0$ 
we show that $\d+\e$ is such a constant. Let $a \in A^{sa}$
with $L_A(a) = 1$.  
Then there is a $b \in B^{sa}$ such that $L_B(b) \leq 1$ and
$\|\iota_A(a) - \iota_B(b)\| \leq \d + \e$, so that
\[L_B(b) \vee (\d+\e)^{-1}N_D(a,b) \leq 1 = L_A(a).
\]
We can interchange the roles of $A$ and $B$. Thus we see that $(\d+\e)^{-1}N_D$
is indeed a bimodule bridge. Notice that for any $a \in A$ and $b \in B$ we have
$N(a^*, b^*) = N(a,b)$. Thus when we define $L$ on $A \oplus B$ by
\[
L(a,b) = L_A(a) \vee L_B(b) \vee (\d+\e)^{-1}N_D(a,b)
\]
it follows from Theorem \ref{th6.3} that $L \in \cM_C(L_A, L_B)$.

But even more is true. As suggested by Li, we will follow the argument in the last 
paragraph of the proof of proposition 4.7 of \cite{Lih2}. Let $\mu \in S(A)$. 
View $A$ and $B$ as subalgebras of $D$ via $\iota_A$ and $\iota_B$.
By the Hahn-Banach theorem, extend $\mu$ to a state $\tilde \nu$ of $D$,
and then restrict $\tilde \nu$ to $B$ to get $\nu \in S(B)$. Then for $a \in A^{sa}$
and $b \in B^{sa}$ we have
\[
|\mu(a) - \nu(b)| = |\tilde \nu(a) - \tilde \nu(b)| \leq \|\iota_A(a) - \iota_B(b)\|
\leq (\d+\e) L(a,b),
\]
where $L$ is defined as above. Consequently if $L(a,b) \leq 1$ then we have
$|\mu(a) - \nu(b)| \leq \d+\e$. Thus $\mu$ is in the $\d+\e$-neighborhood
of $S(B)$ for the metric $\rho_L$ on $S(A \oplus B)$. The same argument
works with the roles of $A$ and $B$ reversed. Since $\e$ is arbitrary, we see 
from this that
\[
\mathrm{prox}(A,B) \leq \mathrm{dist}_{nu}(A,B),
\]
as asserted.

In \cite{Lih3, KrL} Li indicates that $\mathrm{dist}_{nu}$ works very 
well with many of the classes of specific examples whose
metric aspects have been studied. In particular, he 
pointed out to me that $\mathrm{dist}_{nu}$
can be used to give an alternate proof of our Main Theorem (in a qualitative
way). This alternate proof is attractive because of its quite general approach. 
However, a proof via
$\mathrm{dist}_{nu}$ appears to me to be less concrete 
and quantitative than that which we
give in the next sections, both because the proof via $\mathrm{dist}_{nu}$
uses a somewhat deep theorem of Blanchard on the subtrivialization of
continuous fields of nuclear $C^*$-algebras (as discussed in remark
5.5 of \cite{Lih3}), and because of its use of the Hahn-Banach theorem
seen just above. The proof we will give provides specific estimates for the
approximation, and provides a constructive way
of finding a state for one of the algebras that is close to a given state
of the other algebra.
 

\section{Matrix algebras and homogeneous spaces}
\label{sec7}

In this section we begin the study of our main example.  Our
discussion will be fairly parallel to that in \cite{R7} but with some
important differences.  For the reader's convenience we will include
here some fragments of \cite{R7} in order to make precise our setting.
We will usually use the notation used in \cite{R7}.

Let $G$ be a compact group (perhaps even finite at first).  Let $U$ be
an irreducible unitary representation of $G$ on a Hilbert space $\cH$.
Let $B = \cL(\cH)$ denote the $C^*$-algebra of linear operators on
$\cH$ (a ``full matrix algebra'').  There is a natural action, $\a$,
of $G$ on $B$ by conjugation by $U$.  That is, $\a_x(T) = U_xTU_x^*$
for $x \in G$ and $T \in B$.  We introduce metric data into the
picture by choosing a continuous length-function, $\ell$, on $G$.  We
require that $\ell$ satisfy the additional condition that
$\ell(xyx^{-1}) = \ell(y)$ for $x,y \in G$.  This ensures that the
metric on $G$ defined by $\ell$ is invariant under both left and right
translations.  As in Example~\ref{exam2.5} we define a seminorm,
$L_B$, on $B$ by
\[
L_B(T) = \sup\{\|\a_x(T) - T\|/\ell(x): x \ne e_G\}.
\]
Then $L_B$ is a $C^*$-metric on $B$ for the reasons given in
Example~\ref{exam4.3}.

Let $P$ be a rank-one projection in $B$.  Let $H = \{x \in G: \a_x(P)
= P\}$, the stability subgroup for $P$.  Let $A = C(G/H)$, the
$C^*$-algebra of continuous complex-valued functions on $G/H$.  We let
$\l$ denote the usual action of $G$ on $G/H$, and so on $A$, by
translation.  We define a seminorm, $L_A$, on $A$ as in
Example~\ref{exam2.5} by
\[
L_A(f) = \sup\{\|\l_x(f) - f\|/\ell(x): x \ne e_G\}.
\]
Again, $L_A$ is a $C^*$-metric for the reasons given in
Example~\ref{exam4.3}.

We can then ask for estimates of $\prox(A,B)$.  To obtain such an
estimate we need to construct a suitable $C^*$-metric on $A \oplus B$.
We do this as follows.  For any $T \in B$ its Berezin covariant
symbol, $\s_T$, is defined by
\[
\s_T(x) = \tr(T\a_x(P)),
\]
for $x \in G$.  Here $\tr$ is the usual unnormalized trace on $B$.
Because of the definition of $H$ we see that $\s_T \in C(G/H) = A$.
When the $\a_x(P)$'s are viewed as giving states of $B$ via $\tr$ as
above, they form a ``coherent state'', assigning a pure state of $B$
to each pure state of $A$.  Once we note that $\tr$ is $\a$-invariant,
it is easy to see that $\s$ is a unital, positive, norm-non-increasing
$\a$-$\l$-equivariant operator from $B$ to $A$. However eventually one
really wants also the property that if $\s_T = 0$ then $T = 0$.  This
is equivalent to the linear span of the $\a_x(P)$'s in $B$ being all
of $B$.  It is an interesting question as to which representations $U$
admit such a $P$, and how many such $P$'s, even for finite groups.  

We let $\O = \cL(B,A)$, the Banach space of linear operators from $B$
to $A$, equipped with the operator norm corresponding to the
$C^*$-norms on $A$ and $B$.  (Perhaps we should be using the space of
completely bounded operators here.)  We let $M$ and $\L$ denote the
left regular representations of $A$ and $B$.  Then $\O$ is an 
$A$-$B$-bimodule for the operations
\[
f\o = M_f \circ \o \qquad \mbox{ and } \qquad \o T = \o \circ \L_T.
\]
It is easily checked that $\O$ is a normed $A$-$B$-bimodule.  Of
course $\s \in \O$.  We will take our bimodule bridge for $(A,L_A)$
and $(B,L_B)$ to be of the form $(\O,\g^{-1}\s)$ where $\g$ is a
positive real number that is yet to be determined.  Set
\[
N_{\s}(f,T) = \|M_f \circ \s - \s \circ \L_T\|.
\]
Then the seminorm $N$ from $(\O,\g^{-1}\s)$ is defined by
\[
N(f,T) = \g^{-1}N_{\s}(f,T).
\]
We need to determine the values of $\g$ for which $(\O,\g^{-1}\s)$ is
a bimodule bridge so that, in particular, the corresponding seminorm
$L$ has $L_A$ and $L_B$ as quotients for self-adjoint elements.  But,
as a first step in showing what the implication for proximity will be,
we have:

\begin{proposition}
\label{prop7.1}
Suppose that $\g$ is such that $(\O,\g^{-1}\s)$ is a bimodule bridge
for $L_A$ and $L_B$, and let $N$ be the seminorm it determines.  Let
$L = L_A \vee L_B \vee N \vee N^*$ and let $\rho_L$ be the metric on
$S(A \oplus B)$ that $L$ determines.  Then $S(A)$ is in the
$\g$-neighborhood of $S(B)$ for $\rho_L$.
\end{proposition}

\begin{proof}
Let $\mu \in S(A)$.  We must find a $\nu \in S(B)$ such that
$\rho_L(\mu,\nu) \le \g$.  We choose $\nu = \mu \circ \s$.  Let $(f,T)
\in A \oplus B$ be such that $L(f,T) \le 1$, so that $N(f,T) \le 1$
and thus $N_{\s}(f,T) \le \g$.  Then
\begin{align*}
|\mu(f,T)-\nu(f,T)| &= |\mu(f) - \mu(\s_T)| \le \|f-\s_T\| \\
&= \|(M_f \circ \s - \s \circ \L_T)(I)\| \le N_{\s}(f,T) \le \g ,
\end{align*}
where $I$ is the identity element in $B$.  From the definition of
$\rho_L$ it follows that $\rho_L(\mu,\nu) \le \g$.
\end{proof}

We remark that in our earlier paper on ``matrix algebras converge to
the sphere'' \cite{R7} the bridge $N$ that we had used was $N(f,T) =
\g^{-1}\|f-\s_T\|$.
The above calculation reveals that this old $N$ is related to our new
one just by applying our $M_f \circ \s - \s \circ \L_T$ to the
identity operator. The old $N$ is not Leibniz.

To proceed further we now obtain another expression for $N_{\s}$ which
will be more convenient for some purposes.  We note that for $S,T \in
B$ and $f \in A$ we have
\[
(M_f \circ \s - \s \circ \L_T)(S) = f\s_S - \s_{TS},
\]
and that when this is evaluated at $x \in G/H$ we obtain
\begin{align*}
f(x)\s_S(x) - \s_{TS}(x) &= f(x)\tr(S\a_x(P)) - \tr(TS\a_x(P)) \\
&= \tr(\a_x(P)(f(x)I-T)S).
\end{align*}
The operator norm of $M_f \circ \s - \s \circ \L_T$ is then the
supremum of the absolute value of the above expression taken over all
$x \in G/H$ and $S \in B$ with $\|S\| \le 1$.  But $\tr$ gives a
pairing that expresses the dual of $B$ with its operator norm as $B$
with the trace-class norm, which we denote by $\|\cdot\|_1$.  From
this fact we see that
\[
\|M_f \circ \s - \s \cdot \L_T\| = \sup\{\|\a_x(P)(f(x)I-T)\|_1: x \in
G/H\}.
\]
But if $R$ is a rank-one operator then $R^*R = r^2Q$ for some rank-one
projection $Q$ and some $r \in \bR^+$, so that
\[
\|R\|_1 = \tr((R^*R)^{1/2}) = r  = \|R^*R\|^{1/2} = \|R\|,
\]
where the norm on the right is the operator norm.  In this way we
obtain:

\begin{proposition}
\label{prop7.2}
For $f \in A$ and $T \in B$ we have
\[
N_{\s}(f,T) = \sup\{N_x(f,T): x \in G/H\}
\]
where $N_x(f,T) = \|\a_x(P)(f(x)I-T)\|$.
\end{proposition}

We remark that $N_x(f,T)$ can easily be checked to be strongly
Leibniz.


\section{The choice of the constant $\g$}
\label{sec8}

Let us first see what choices of $\g$ ensure that $L$ has $L_A$ as a
quotient.  It suffices to choose $\g$ such that for any $f \in A$ we
can find $T \in B$ such that $L_B(T) \vee N(f,T) \le L_A(f)$.  On
$G/H$ let us momentarily use the $G$-invariant measure of mass $1$ to
give $A$ the norm from $L^2(G/H)$.  Similarly, on $B$ we put the
Hilbert--Schmidt norm from the {\em normalized} trace.  Then $\s$ has
an adjoint operator, which we denote by $\breve \s$.  It is easily
computed \cite{R7} to be defined by
\[
{\breve \s}_f = d \int_{G/H} f(x)\a_x(P)dx,
\]
where $d$ is the dimension of $\cH$.
One can easily verify that $\breve \s$ is a positive and
$\l$-$\a$-equivariant map from $A$ to $B$.  Furthermore, ${\breve
\s}_1 = d \int \a_x(P)dx$, which is clearly $\a$-invariant, and so is
a scalar multiple of $I$ since $U$ is irreducible.  But clearly the
usual trace of $d \int \a_x(P)dx$ is $d$.  Thus ${\breve \s}_1 = I$,
that is, $\breve \s$ is unital.  (This is why we used the normalized
traces in defining $\breve \s$.)  It follows that $\breve \s$ is also
norm non-increasing.

Then, given $f \in A$, we will choose $T$ to be $T = {\breve \s}_f$.
It is easily seen (as in the proof of proposition~$1.1$ of \cite{R7})
that $L_B({\breve \s}_f) \le L_A(f)$.  For any $x \in G/H$ we have by
equivariance of $\breve \s$
\[
N_x(f,{\breve \s}_f) = \|\a_x(P)(f(x)I - {\breve \s}_f)\| =
\|P((\l_x^{-1}f)(e)I - {\breve \s}_{\l_x^{-1}f})\|.
\]
Since $f$ is arbitrary and $L_A$ is $\l$-invariant, it suffices for us
to consider $\|P(f(e)I - {\breve \s}_f)\|$.  But 
\begin{align*}
\|P(f(e)I-{\breve \s}_f)\| &= \left\|P\left(f(e)d \int \a_y(P)dy - d
\int f(y)\a_y(P)dy\right)\right\| \\
&= d\left\| \int (f(e)-f(y))P\a_y(P)dy\right\| \\
&\le L_A(f) \ d \int \rho_{G/H}(e,y)\|P\a_y(P)\|dy,
\end{align*}
where $\rho_{G/H}$ is the ordinary metric on $G/H$ from $L_A$.    From all of this we
obtain:

\begin{proposition}
\label{prop8.1}
Set $\g^A = d \int \rho_{G/H}(e,y)\|P\a_y(P)\|dy$.  Then for any $\g
\ge \g^A$ the seminorm $L = L_A \vee L_B \vee \g^{-1}(N_{\s} \vee
N_{\s}^*)$ on $A \oplus B$ has $L_A$ as its quotient on $A$.
\end{proposition}

We remark that in the above proposition we do  not have to restrict
attention to self-adjoint elements, in contrast to the requirement in
Definition~\ref{def6.2}.  Note
that ${\breve \s}_{\bar f} = ({\breve \s}_f)^*$. 
I do not know whether the above condition on
$\g$ is the best that can be obtained in the absence of further
hypotheses on $G$, $U$, $P$ and $\ell$.

We now consider the quotient of $L$ on $B$.  Given $T \in B$ we seek
$f \in A$ such that $L_A(f) \vee N(f,T) \le L_B(T)$.  We choose $f =
\s_T$, and seek what requirement this puts on $\g$.  As above, it is
easy to check that $L_A(\s_T) \le L_B(T)$.  Again by equivariance we
have
\[
N_x(\s_T,T) = \|\a_x(P)(\tr(T \a_x(P))I - T)\| =
\|P(\tr(P\a_x^{-1}(T))I-\a_x^{-1}(T))\|.
\]
Since $T$ is arbitrary and $L_B$ is $\a$-invariant, it suffices to
choose $\g$ large enough that $\|P\tr(PT)-PT\| \le \g L_B(T)$ for all
$T \in B$.  Notice that the left-hand side gives a seminorm (with
value $0$ for $T = P$ or $I$) on the quotient space ${\tilde B} =
B/\bC I$, while $L_B$ gives a norm on ${\tilde B}$.  Since $B$ is
finite-dimensional, there does exist a finite $\g$ such that the above
inequality is satisfied.  Notice also that $\s_{T^*} = (\s_T)^-$.
Thus we obtain:

\begin{proposition}
\label{prop8.2}
Define $\g^B$ by
\[
\g^B = \sup\{\|P\tr(PT) - PT\|: T \in B \mbox{ and } L_B(T) \le 1\}.
\]
Then $\g^B$ is finite, and for any $\g \ge \g^B$ the seminorm 
$L = L_A\vee L_B\vee \g^{-1}(N_{\s}\vee N_{\s}^*)$ on $A \oplus B$ has
$L_B$ as its quotient on $B$.
\end{proposition}

For later use we now express $\|P(\tr(PT)I - T)\|$ in a different
form.  Since taking adjoints is an isometry, and by the
$C^*$-relation, and by the fact that if $R$ is a positive operator then
$\|PRP\| = \tr(PRP)$ because $P$ is of rank~$1$, we have
\begin{align*}
\|P(\tr(PT)I - T)\|^2 &= \|P(\tr(PT)I-T)(\tr(PT)I-T)^*P\| \\
&= \tr\Big( P(\tr(PT)I-T)(\tr(PT)I-T)^*P)\Big) \\
&= |\tr(PT)|^2 - \tr(PTP)\overline{\tr}(PT) \\
&\quad - \tr(PT)\tr(PT^*P) + \tr(PTT^*P) \\
&= \tr(PTT^*P) - |\tr(PT)|^2.
\end{align*}
Thus:

\begin{proposition}
\label{prop8.3}
For any $T \in B$ we have
\[
\|P(\tr(PT)I-T)\| = (\tr(PTT^*P) - |\tr(PT)|^2)^{1/2}.
\]
\end{proposition}

We remark that if $\xi$ is a unit vector in the range of $P$ then
\[
\tr(PTT^*P) - |\tr(PT)|^2 = \<TT^*\xi,\xi\> - |\<T^*\xi,\xi\>|^2.
\]
When $T$ is self-adjoint this is the ``mean-square deviation'' of $T$
in the state determined by $\xi$ \cite{Thr3}.

We now need to consider how small a neighborhood of $S(A)$ contains
$S(B)$.  Let $\nu \in S(B)$ be given.  We choose $\mu = \nu \circ
{\breve \s}$, and observe that $\mu \in S(A)$.  Let $(f,T) \in A
\oplus B$ be such that $L(f,T) \le 1$, so that $N_{\s}(f,T) \le \g$.
Then
\begin{align*}
|\mu(f,T)-\nu(f,T)| &= |\nu({\breve \s}_f)-\nu(T)| \le \|{\breve
\s}_f-T\| \\
&= \left\| d \int f(x)\a_x(P)dx - d \int \a_x(P)Tdx\right\| \\
&= d\left\| \int \a_x(P)(f(x)I-T)dx\right\| \le d \int N_x(f,T)dx \\
&\le dN_{\s}(f,T) \le d\g.
\end{align*}
But the presence of $d$ here causes us difficulties later, so we take
another path, namely that used near the end of section~$2$ of
\cite{R7}.  We have
\begin{align*}
\|{\breve \s}_f-T\| &\le \|{\breve \s}_f - {\breve \s}(\s_T)\| +
\|{\breve \s}(\s_T) - T\| \\
&\le \|f - \s_T\| + \|{\breve \s}(\s_T) - T\| \le \g^A + \|{\breve
\s}(\s_T)-T\|,
\end{align*}
where we have used that $\|f - \s_T\| \leq N_\s(f, T)$, as seen in the
proof of Proposition \ref{prop7.1}.
Notice that $T \mapsto \|{\breve \s}(\s_T)-T\|$ is a seminorm on $B$
which takes value $0$ for $T = I$, and so drops to a seminorm on
${\tilde B} = B/\bC I$, where $L_B$ becomes a norm.

\begin{notation}
\label{note8.4}
We set
\[
\d^B = \sup\{\|T - {\breve \s}(\s_T)\|: L_B(T) \le 1\}.
\]
\end{notation}

With this notation the above discussion gives:

\begin{proposition}
\label{prop8.5}
Suppose that $\g \geq \g^A \vee \g^B$, so that $L$ has $L_A$ and $L_B$
as quotients (where $L = L_A \vee L_B \vee \g^{-1}(N_{\s} \vee
N_{\s}^*)$).  Then $S(B)$ is in the $(\g^A+\d^B)$-neighborhood of
$S(A)$.
\end{proposition}


\section{The set-up for compact Lie groups}
\label{sec9}

We now specialize to the case in which $G$ is a compact connected
semisimple Lie group.  We use many of the techniques used in
sections~$6$ and $7$ of \cite{R7}, and we usually use the notation
established in sections~$5$ and $6$ of \cite{R7}.  We now review that
notation.  We let $\fG_0$ denote the Lie algebra of $G$, while $\fG$
denotes the complexification of $\fG_0$.  We choose a maximal torus in
$G$, with corresponding Cartan subalgebra of $\fG$, its set of roots,
and a choice of positive roots.  We let $(U,\cH)$ be an irreducible
unitary representation of $G$, and we let $U$ also denote the
corresponding representation of $\fG$.  We choose a highest weight
vector, $\xi$, for $(U,\cH)$ with $\|\xi\| = 1$.  For any $n \in \bZ_{\ge 1}$ we set
$\xi^n = \xi^{\otimes n}$ in $\cH^{\otimes n}$, and we let
$(U^n,\cH^n)$ be the restriction of $U^{\otimes n}$ to the 
$U^{\otimes n}$-invariant subspace, $\cH^n$, of $\cH^{\otimes n}$ which is
generated by $\xi^n$.  Then $(U^n,\cH^n)$ is an irreducible
representation of $G$ with highest weight vector $\xi^n$, and its
highest weight is just $n$ times the highest weight of $(U,\cH)$.  We
denote the dimension of $\cH^n$ by $d_n$.

We let $B^n = \cL(\cH^n)$.  The action of $G$ on $B^n$ by conjugation
by $U^n$ will be denoted simply by $\a$.  We assume that a continuous
length function, $\ell$, has been chosen for $G$, and we denote the
corresponding $C^*$-metric on $B^n$ by $L_n^B$.  We let $P^n$ denote
the rank-one projection along $\xi^n$.  Then the $\a$-stability
subgroup $H$ for $P = P^1$ will also be the stability subgroup for
each $P^n$. Let $\g_n^A$ and $\g_n^B$ be the constants defined 
in Propositions \ref{prop8.1} and \ref{prop8.2} but for $P^n$.

As done earlier, we let $A = C(G/H)$, and we let $L_A$ be the seminorm
on $A$ for $\ell$ and the action of $G$.  We can now state the main
theorem of this paper.

\begin{theorem}
\label{th9.1}
Let notation be as above.  Set $\g_n = \max\{\g_n^A, \g_n^B\}$ for each $n$,
and let $L_n$ be defined on $A \oplus B^n$ as in Proposition \ref{prop7.1} but using $\g_n$.
Then $L_n \in \cM_C(L_A, L_B)$, and the sequence $\{L_n\}$ shows
that the sequence $\{\prox(A,B^n)\}$
converges to $0$ as $n$ goes to $\i$.
\end{theorem}

The next three sections will be devoted to the proof of this theorem.


\section{The proof that $\g_n^A \to 0$}
\label{sec10}

Consistent with the notation of Proposition~\ref{prop8.1}, we have set
\[
\g_n^A = d_n\int \rho_{G/H}(e,x)\|P^n\a_x(P^n)\|dx.
\]

\begin{proposition}
\label{prop10.1}
The sequence $\{\g_n^A\}$ converges to $0$.
\end{proposition}

\begin{proof}
For any two vectors $\eta,\z$ we let $\<\eta,\z\>_0$ denote the
rank-one operator that they determine.  Then for any $n$ we have
\begin{align*}
\|P^n\a_x(P^n)\| &= \|\<\xi^n,\xi^n\>_0\<U_x^n\xi^n,U_x^n\xi^n\>_0\|
\\
&= |\<U_x^n\xi^n,\xi^n\>| = |\<U_x\xi,\xi\>|^n = \|P\a_x(P)\|^n.
\end{align*}
We use the analogous treatment given in lemma~$3.3$ and
theorem~$3.4$ of \cite{R7}, where we see that
$d_n|\<U_x\xi,\xi\>|^{2n}dx$ ($= d_n\|P^n\a_x(P^n)\|^2dx$) is a
probability measure on $G/H$, and that the sequence of these
probability measures converges in the weak-$*$ topology to the
$\d$-measure on $G/H$ supported at $eH$.  Since
$\rho_{G/H}(e,e) = 0$, it follows that the sequence $d_n \int
\rho_{G/H}(e,x)\|P^n\a_x(P^n)\|^2dx$ converges to $0$ .  Now
\begin{align*}
\g_{2n}^A = &d_{2n} \int \rho_{G/H}(e,x)\|P\a_x(P)\|^{2n}dx \\ &=
(d_{2n}/d_n)d_n \int \rho_{G/H}(e,x)\|P^n\a_x(P^n)\|^2dx,
\end{align*}
and so if we can show that $(d_{2n}/d_n)$ is bounded, then we find
that the sequence $\{\g_{2n}\}$ converges to $0$.  We use the Weyl
dimension formula, as presented for example in theorem~$4.14.6$ of
\cite{Vrd}, to show that $\{d_{2n}/d_n\}$ is bounded.  We let $\o$ be
the highest weight of $U$ for our choice $\cP$ of positive roots.  If
one examines the dimension formula, it is evident that one only needs
to use those positive roots $\a$ such that $\<\o,\a\> > 0$.  We denote
this set by $\cP_{\o}$, and we denote its cardinality by $p$.  It is
clear that for any $n \in \bZ_{>0}$ we have $\cP_{n\o} = \cP_{\o}$.
The Weyl dimension formula then tells us that
\[
d_n = \left.\left(\prod\<n\o + \d,\a\>\right)\right
/\left(\prod\<\d,\a\>\right)
\]
where both products are taken over $\cP_{\o}$, and $\d$ is half the
sum of the positive roots.  Thus
\begin{align*}
d_{2n}/d_n &= \left.\left(\prod\<2n\o + \d,\a\>\right)\right
/\left(\prod\<n\o+\d,\a\>\right) \\
&= \prod(1 + \<n\o,\a\>/\<n\o+\d,\a\>) \le  2^p,
\end{align*}
so that the sequence $d_{2n}/d_n$ is bounded as needed, and
consequently the sequence $\{\g_{2n}^A\}$ converges to $0$.  In the
same way, we find that $d_{n+1}/d_n \le (1+n^{-1})^p$.  Since $0 \le
\|P\a_x(P)\| \le 1$, we have $\|P\a_x(P)\|^n \ge \|P\a_x(P)\|^{n+1}$.
Thus the integrals defining $\g_n^A$ are non-increasing.  It follows
that $\g_{2n+1}^A \le (1+(2n)^{-1})^p\g_{2n}^A$.  Since the sequence
$\{\g^A_{2n}\}$ converges to $0$, it follows that the sequence
$\{\g^A_{2n+1}\}$ does also, so that the sequence $\{\g_n^A\}$
converges to $0$.
\end{proof}


\section{Properties of Berezin symbols}
\label{sec11}

We now need results related to those given in sections 4 and 5 of
\cite{R7}, leading to the proof of theorem 6.1 of \cite{R7}, and we will
shortly also need theorem 6.1 of \cite{R7} itself. But Jeremy Sain has found
a substantial simplification of the proof of theorem 6.1 of \cite{R7}.
He gives his argument in section 4.4 of \cite{Sai} in the more
complicated context of quantum groups. We will use his arguments here
in our present context. This will in particular provide Sain's proof of
theorem 6.1 of \cite{R7}.

As in \cite{R7}, we denote the Berezin symbol map
from $B^n$ to $A = C(G/H)$ by $\s^n$.  From theorem~$3.1$ of \cite{R7}
we find that $\s^n$ is injective because $\xi^n$ is a highest weight
vector. Consistent with the notation defined near the beginning of
Section 8, we denote the adjoint of $\s^n$ by ${\breve \s}^n$.
We let
\begin{equation} \label{eq11.1}
\d_n^A = \int_{G/H} \rho_{G/H}(e, x) d_n \mathrm{tr}(P^n \a_x(P^n)) \ dx .
\end{equation}
In section 3 of \cite{R7} $\d_n^A$ was denoted by $\g_n$, and 
theorem 3.4 of \cite{R7} shows both that the sequence $\{\d_n^A\}$
converges to 0, and that 
\begin{equation} \label{eq11.2}
\|f - \s^n(\breve \s^n(f))\|_\infty \ \leq \ \d_n^AL_A(f)
\end{equation}
for all $f \in A$ and all $n$. We remark that $\s^n \circ \breve \s^n$
is often called the ``Berezin transform'' (for a given $n$).

As in section~$4$ of \cite{R7} we let ${\hat G}$ denote the set of
equivalence classes of irreducible unitary representations of $G$.
For any finite subset $\cS$ of ${\hat G}$ we let $A_{\cS}$ and
$B_{\cS}^n$ denote the direct sum of the isotypic components of $A$
and $B^n$ for the representations in $\cS$ and for the actions of $G$
on $A$ and $B^n$ (and similarly for actions on other Banach spaces).
Since $\s^n$ is equivariant, it carries $B_{\cS}^n$ into $A_{\cS}$.
Since $\s^n$ is injective, it follows that the dimension of
$B_{\cS}^n$ is no larger than that of $A_{\cS}$, which is finite.

Since $\{\d_n^A\}$ converges to 0, it follows from equation \ref{eq11.2}
that $\s^n \circ \breve \s^n$ converges strongly to the identity operator 
on the space of functions $f$ for which $L_A(f) < \infty$. But $A_\cS$ is 
contained in this space and is finite-dimensional, and $\s^n \circ \breve \s^n$
carries $A_\cS$ into itself for each $n$. Consequently $\s^n \circ \breve \s^n$
restricted to $A_\cS$ converges in norm to the identity operator on
$A_\cS$. It follows that there is an integer, $N_\cS$, such that
$\s^n \circ \breve \s^n$ on $A_\cS$ is invertible and
$\|(\s^n \circ \breve \s^n)^{-1}\| < 2$ for every $n > N_\cS$. In particular, 
$\s^n$ from $B^n_\cS$ to
$A_\cS$ will be surjective for $n > N_\cS$. Since, as mentioned
above, $\s^n$ is always injective, and $\|\s^n\| = 1 = \|\breve\s^n\|$
for all $n$, we can quickly see that:

\setcounter{theorem}{2}
\begin{lemma}
\label{lem11.3}
(See corollary 4.17 of \cite{Sai}.) For $n > N_\cS$ both $\s^n$ 
and $\breve \s^n$ going between $A_\cS$ and $B^n_\cS$ are invertible 
and their inverses have operator-norm no bigger than 2.
\end{lemma}

Fix $n > N_\cS$, and let $T \in B^n_\cS$ be given.
Set $f = (\breve\s^n)^{-1}(T)$. Note that $f$ is well-defined,
and that $\|f\|_\infty \leq 2\|T\|$ by Lemma \ref{lem11.3}. Then
\[
\|T - \breve\s^n(\s^n_T)\| = \|\breve\s^n(f) -  \breve\s^n(\s^n(\breve\s^n_f))\|  
\leq  \|f - \s^n(\breve\s^n_f)\| \leq \d_n^A L_A(f) ,
\]
where we have used inequality \ref{eq11.2} for the last inequality above.
Because $(\breve\s^n)^{-1}$ is $\a$-$\l$-equivariant and
$\|(\breve\s^n)^{-1}\| \leq 2$, we have $L_A(f) \leq 2L_n^B(T)$.
We have thus obtained:

\begin{lemma}
\label{lem11.4}
(See proposition 4.19 of \cite{Sai}.) For any $n > N_\cS$ and any
$T \in B^n_\cS$ we have
\[
\|T - \breve\s^n(\s^n_T)\| \leq 2\d_n^AL^B_n(T).
\]
\end{lemma}

Choose a faithful finite-dimensional unitary representation, $\pi_0$,
of $G$ that contains the trivial representation, and let $\pi = \pi_0
\otimes {\bar \pi}_0$, where ${\bar \pi}_0$ is the contragradient
representation for $\pi_0$.  Let $\chi$ be the character of $\pi$.
Then $\chi$ is a non-negative real-valued function on $G$.  Since
$\pi$ is faithful, we have the strict inequality $\chi(x) < \chi(e)$
for any $x \in G$ with $x \ne e$.  Let $\chi^m$ denote the character
of $\pi^{\otimes m}$, so that equally well it is the $m^{\mbox{th}}$
pointwise power of $\chi$.  Set $\var_m = \chi^m/\|\chi^m\|_1$.  Then
the sequence $\{\var_m\}$ is a norm-$1$ approximate identity for the
convolution algebra $L^1(G)$, as seen in the proof of theorem~$8.2$ of
\cite{R6}.  Furthermore, each $\var_m$ is central in $L^1(G)$.  Let
$\b$ be an isometric strongly continuous action of $G$ on a Banach
space $D$, and let $L^D$ be the corresponding seminorm for $\ell$.
Let $\b_{\var_n}$ denote the corresponding ``integrated form''
operator.
As in the proof of lemma~$8.3$ of \cite{R6}, for each $d \in D$ we have
\begin{eqnarray*}
\|d-\b_{\var_m}(d)\|&=& \|d\int \var_m(x) \ dx \ - \ \int\var_m(x)\b_x(d) \ dx\|  \\
&\le&\int \var_m\|d - \b_x(d)\|dx \le  \left( \int \var_m(x)\ell(x)dx\right)L^D(d) ,
\end{eqnarray*}
and the sequence $\left\{ \int
\var_m(x)\ell(x)dx\right\}$ converges to $0$.

We can now argue exactly as in the rest of the proof of theorem 6.1
of \cite{R7} to obtain:

\begin{theorem}
\label{thm11.5}
(Theorem 6.1 of \cite{R7})
For each $n \ge 1$ let $\d_n^B$ be as defined in Notation \ref{note8.4}
but for $B^n$, so that it is
the smallest constant such that
\[
\|T - {\breve \s}^n(\s_T^n)\| \le \d^B_nL_n^B(T)
\]
for all $T \in B^n$.  Then the sequence $\{\d^B_n\}$ converges to $0$.
\end{theorem}

\begin{proof}[Proof of Theorem \ref{thm11.5}]
Let $\e > 0$ be given.  We can choose $\var =\var_m$ as just above
such that for any ergodic action $\b$ of
$G$ on any unital $C^*$-algebra $C$  we have $\|c - \b_{\varphi}(c)\| \le (\e/3)L(c)$
for all $c \in C$. Now $\var$ is a positive function, and is a
linear combination of the characters of a finite subset $\cS$ of $\hat G$, 
and so the integrated operator $\b_{\varphi}$
is a completely positive unital equivariant map of $C$ onto its
$S$-isotypic component.

 Then for every $n$ and every $T \in B^n$ we have $\a_{\varphi}(T) \in B^n_\cS$ and
\[
\|T - {\breve \s}^n(\s_T^n)\| \le (\e/3)L_n^B(T) + \|\a_{\varphi}(T) -
{\breve \s}^n(\s_{\a_{\varphi}(T)}^n)\| + (\e/3)L_n^B(T).
\]
From Lemma \ref{lem11.4} there is an integer
$N_\e$ such that for any $n >N_\e$ and any $T' \in  B^n_\cS$ we have
\[
\|T' - \breve\s^n(\s^n( T'))\| \leq (\e/3)L^B_n(T').
\]
Since $\a_{\varphi}(T) \in B^n_\cS$, we can apply this
 to $T' = \a_{\varphi}(T)$. When we use the fact that
$L^B_n(\a_{\varphi}(T)) \le L^B_n(T)$, we see that for
any $n > N_\e$ and any $T \in B^n$ we have
\[
\|T - {\breve \s}^n(\s_T^n)\| \le \e L_n^B(T).
\]
This immediately implies the statement about the sequence $\{\d^B_n\}$.

\end{proof}


\section{The proof that $\g_n^B \to 0$}
\label{sec12}

Consistent with the notation of Proposition~\ref{prop8.2}, we have set
\[
\g_n^B = \sup\{\|P^n\tr(P^nT) - P^nT\|: T \in B^n \mbox{ and }
L_n^B(T) \le 1\}.
\]

\begin{proposition}
\label{prop12.1}
The sequence $\{\g_n^B\}$ converges to $0$.
\end{proposition}

\begin{proof}
Let $\e > 0$ be given.  With the notation that we used just before
Theorem~\ref{thm11.5},
choose $m_0$ such that for $\var = \var_{m_0}$ we have $\int
\var(x)\ell(x)dx \le \e/4$.  Then by the calculation done there we have
\[
\|T - \a_{\var}(T)\| \le (\e/4)L_n^B(T)
\]
for all $n$ and for all $T \in B^n$.
Then for any $n$ and any $T \in B^n$
\begin{align*}
&\|(P^n\tr(P^nT)-P^nT) - (P^n\tr(P^n\a_{\var}(T))
- P^n\a_{\var}(T))\| \\
&\le |\tr(P^n(T-\a_{\var}(T)))| + \|T-\a_{\var}(T)\| \\
&\le 2\|T-\a_{\var}(T)\| \le (\e/2) L_n^B(T),
\end{align*}
where for the next-to-last inequality we have used the fact that
$P^n(T-\a_{\var}(T))$ is of rank~$1$.

Now as discussed in the proof of Theorem~\ref{thm11.5}, $\var$ is a 
linear combination of the characters of a finite subset $\cS$ of $\hat G$. 
Thus $\a_{\var}(T) \in B_{\cS}^n$ and 
$L_n^B(\a_{\var}(T)) \leq L_n^B(T)$, and so we now see that it suffices to
prove:

\begin{mainlemma}
\label{mainlem12.2}
Let $\cS$ be given.  For any $\e > 0$ there is an integer $N_{\e}$ such that for any $n \ge N_{\e}$ and any
$T \in B_{\cS}^n$ we have
\[
\|P^n\tr(P^nT) - P^nT\| \le (\e/2)L_n^B(T).
\]
\end{mainlemma}

\begin{proof}
Let $f \in A$,
and let $n$ be given. Because $A$ is commutative and $\breve\s^n$
is positive, it follows from Kadison's generalized Schwarz inequality
(e.g. 10.5.9 of \cite{KdR}) that we have
\[
\breve\s^n_f(\breve\s^n_f)^* \ \leq \ \breve\s^n_{f \bar f}
\]
for the usual order on positive operators. When we combine this
with  Proposition~\ref{prop8.3}  we obtain
\begin{align*}
\|P^n(\tr&(P^n\breve\s^n_f)I - \breve\s^n_f)\|^2 = 
\tr(P^n\breve\s^n_f(\breve\s^n_f)^*P^n)
- |\tr(P^n\breve\s^n_f)|^2 \\
&\leq \mathrm{tr}(P^n\breve\s^n_{f\bar f}) - |\tr(P^n\breve\s^n_f)|^2 
 = (\s^n(\breve\s^n_{f\bar f}))(e) - |\s^n(\breve\s^n_f)(e)|^2  ,
\end{align*}
which by equation~\ref{eq11.2} above and theorem 3.4 of \cite{R7} converges to
\[
(f \bar f)(e) - |f(e)|^2 =  0
\]
as $n$ increases.

For each $n$ define an operator, $J^n$, on $B^n$ by
\[
J^n(T) = P^n(\mathrm{tr}(P^nT)I \ - \ T) .
\]
The calculation above shows that the sequence $J^n(\breve\s^n_f)$ 
converges to 0 for any $f \in A$ with $L^A(f) < \infty$. For $\cS$
as above it follows that the sequence of restrictions of
$J^n \circ \breve\s^n$ to $A_\cS$ converges to 0 in operator norm. 
Let $N_\cS$ be as in Lemma~\ref{lem11.3}, so that
$\|(\breve\s^n)^{-1}\| \leq 2$ for $n >N_\cS$. It follows that for $n > N_\cS$
we have $\|J_n\| \leq 2\|J^n \circ \breve\s^n\|$, so that the 
sequence of restrictions of $J^n$ to $B^n_\cS$ converges
to 0 in norm. Thus for any $\e' > 0$ we can find an $n_{\e'}$ such
that for $n >n_{\e'}$ and all $T \in B_{\cS}^n$ we have
\[
\|J^n(T)\| \le \e'\|T\|.
\]
Now $J^n(I) = 0$, and so it follows that
\[
\|J^n(T)\| \le \e'\|{\tilde T}\|^{\sim},
\]
where much as before  $\|\cdot\|^{\sim}$ denotes the quotient norm on
${\tilde B^n}=  B^n/\bC I$. 
But by lemma~$2.4$ of \cite{R4} the radius of each of
the algebras $B^n$ is no larger than $r = \int \ell(x)dx$, in the
sense that $\|{\tilde T}\|^{\sim} \le rL_n^B(T)$ for all $T \in B^n$. We include a slightly simpler proof here. For $T \in B^n$ let
$\eta(T) = \int \a_x(T) \ dx$, so that $\eta(T) \in {\mathbb C}I$
since $U^n$ is irreducible. Then
\[
\|{\tilde T}\|^{\sim} \leq \|T - \eta(T)\| = 
\| \int (T - \a_x(T)) dx \|
\leq L_n^B(T) \int \ell(x) dx.
\]
It
follows that
\[
J^n(T) \le r\e'L_n^B(T).
\]
Consequently, if we choose $\e' = \e/(2r)$, and set $N_{\e} = n_{\e'} \vee
N_{\cS}$, we find that for $n \ge N_{\e}$ we have
\[
\|P^n\tr(P^nT) - P^nT\| \le (\e/2)L_n^B(T)
\]
for all $T \in B_{\cS}^n$, as needed.
\end{proof}
\end{proof}


\section{The proof of the main theorem}
\label{sec13}

We now use the results of the previous sections to prove
Theorem~\ref{th9.1}.  For any $n$ set $\g_n = \max(\g_n^A,\g_n^B)$,
and define $L_n$ on $A \oplus B^n$ by
\[
L_n(f,T) = L_A(f) \vee L_n^B(T) \vee \g_n^{-1}(N_{\s^n}(f,T) \vee
N_{\s^n}({\bar f},T^*)).
\]
Then for each $n$ we have $\g_n \ge \g_n^A$ so that the quotient of
$L_n$ on $A$ is $L_A$ by Proposition~\ref{prop8.1}, and we have $\g_n \ge
\g_n^B$ so that the quotient of $L_n$ on $B^n$ is $L_n^B$ by
Proposition~\ref{prop8.2}.  Thus $L_n$ is in $\cM_C(L_A,L_n^B)$ as
defined in Notation~\ref{note5.5}. 

Then according to Proposition~\ref{prop7.1} (with notation as in
Proposition~\ref{prop8.1} and in the sentence before
Proposition~\ref{prop10.1}), $S(A)$ is in the $\g_n$-neighborhood
of $S(B^n)$ for  $\rho_{L_n}$. Furthermore, according to Proposition~\ref{prop8.5}
(with notation as in Theorem~\ref{thm11.5})  $S(B^n)$ is in the
$(\g^A_n + \d^B_n)$-neighborhood of $S(A)$. It follows that 
\[
\dist_H^{\rho_{L_n}}(S(A), S(B^n)) \leq \max\{\g^A_n + \d^B_n, \g^n\} 
\leq \max\{\g^A_n + \d^B_n, \g^B_n\} ,
\]
and so
\[
\prox(A,B^n) \leq \max\{\g^A_n + \d^B_n, \g^B_n\} .
\]
But $\g^A_n$, $\d^B_n$ and $\g^B_n$ all converge to 0 as $n$ goes
to $\infty$, according to Proposition~\ref{prop10.1}, 
Theorem \ref{thm11.5} (theorem~$6.1$ of \cite{R7}), and 
Proposition~\ref{prop12.1} respectively. Consequently
$\prox(A,B^n)$ converges to 0 as $n$ goes to $\infty$, as desired.


\section{Matricial seminorms}
\label{sec14}

In this section we will briefly describe the relations between the
previous sections of this paper and several variants of quantum
Gromov--Hausdorff distance.

The first variant is the matricial quantum Gromov--Hausdorff
distance introduced by Kerr \cite{Krr}.  It has the advantage that if
two $C^*$-algebras with $\Lip$-norms are at distance $0$ for his
distance then the $C^*$-algebras are isomorphic.  We will not repeat
here Kerr's definitions and results for general operator systems;
rather we will only indicate, somewhat sketchily, what Kerr's variant
says in the context of the present paper.  For any unital
$C^*$-algebra $A$ and each $q \in \bZ_{>0}$ the $*$-algebra $M_q(A)$
of $q \x q$ matrices with entries in $A$ has a unique $C^*$-norm.  The
collection of these $C^*$-norms forms a ``matricial norm'' for $A$.
Given unital $C^*$-algebras $A$ and $B$, a linear map $\var: A \to B$
determines for each $q$ a linear map, $\var^q$, from $M_q(A)$ to
$M_q(B)$, by entry-wise application.  One says that $\var$ is
``completely positive'' if each $\var^q$ is positive as a map between
$C^*$-algebras.  For each $q$ let $UCP_q(A)$ denote the collection of
unital completely positive maps from $A$ into $M_q(\bC)$.  The
$UCP_q(A)$'s are called the ``matricial state-spaces'' of $A$.  All
these considerations apply equally well to unital $C^*$-normed
algebras, where ``positive'' is with respect to the completions.

Let a $\Lip^*$-norm, $L$, on $A$ be specified.  Then Kerr defines a
metric, $\rho_L^q$, on $UCP_q(A)$ by
\[
\rho_L^q(\var,\psi) = \sup\{\|\var(a)-\psi(a)\|: L(a) \le 1\},
\]
and he shows that the topology on $UCP_q(A)$ determined by $\rho_L^q$
agrees with the point-norm topology (and so is compact).  Now let
$(A,L_A)$ and $(B,L_B)$ be unital $C^*$-algebras with $\Lip^*$-norms.
Essentially as in definition~$4.2$ of \cite{R6} let $\cM(L_A,L_B)$
denote the set of $\Lip^*$-norms on $A \oplus B$ whose quotients on
the self-adjoint part agree with $L_A$ and $L_B$.  Note that
$UCP_n(A)$ and $UCP_n(B)$ can be viewed as subsets of $UCP_n(A \oplus
B)$ in an evident way.  Then for each $q$ Kerr defines the
$q$-distance, $\dist_s^q$, between $A$ and $B$ by
\[
\dist_s^q(A,B) = \inf\{\dist_H^{\rho_L^q}(UCP_q(A),UCP_q(B)): L \in
\cM(A,B)\},
\]
and he defines the complete distance, $\dist_s$, by
\[
\dist_s(A,B) = \sup_q\{\dist_s^q(A,B)\}.
\]
Finally (for our purposes), he shows that for our setting of coadjoint
orbits with $A = C(G/H)$ and $B^n = \cL(\cH_n)$ with their
$\Lip^*$-norms from a length function $\ell$, one has
\[
\lim_{n \to \i} \dist_s(A,B^n) = 0.
\]

We can quickly adapt Kerr's arguments to our Leibniz setting.  For
$C^*$-algebras $A$ and $B$ equipped with $C^*$-metrics, we define
$\cM_C(L_A,L_B)$ exactly as in Notation~\ref{note5.5}.  Any $L$ in
$\cM_C(L_A,L_B)$ is, in particular, a $\Lip^*$-norm, and so defines
for each $q$ the metric $\rho_L^q$ on $UCP_q(A \oplus B)$.  We can
then define, for each $q$,
\[
\prox^q(A,B) = \inf\{\dist_H^{\rho_L^q}(UCP_q(A),UCP_q(B)): L \in
\cM_C(A \oplus B)\}.
\]
Then we  can define ``complete proximity'' by
\[
\prox_s(A,B) = \sup_q\{\prox^q(A,B)\}.
\]
Of course, we have
\[
\dist_s(A,B) \le \prox_s(A,B).
\]

\begin{theorem}
\label{th14.1}
For $A = C(G/H)$ and $B^n = \cL(\cH_n)$ with their $C^*$-metrics $L_A$
and $L_B^n$ as defined earlier in terms of a length function on $G$,
we have
\[
\lim_{n \to \i} \prox_s(A,B^n) = 0.
\]
\end{theorem}

\begin{proof}
We follow the outline of Kerr's example~$3.13$ of \cite{Krr}, but for
a given $n$ we set, as earlier,
\[
L_n = L_A \vee L_n^B \vee N_n \vee N_n^*
\]
with $N_n = \g_n^{-1}N_{\s^n}$ and with $\g_n$ chosen exactly as in
the proof of Theorem~\ref{th9.1} that is completed in Section~\ref{sec13}.
Thus $L_n \in \cM_C(L_A,L_n^B)$.  The key observation, for Kerr and
for us, is that $\s^n$ and ${\breve \s}^n$ are (unital) completely
positive maps, so that if $\var \in UCP_q(A)$ then $\var \circ \s^n$
is in $UCP_q(B^n)$, and similarly for ${\breve \s}^n$.  Given $\var
\in UCP_q(A)$, set $\psi = \var \circ \s^n$.  Then exactly as in the
proof of Proposition~\ref{prop7.1} we see that if $L_n(f,T) \le 1$
then
\[
\|\var(f)-\psi(T)\| \le \|f-\s_T\| \le \g_n,
\]
so that $UCP_q(A)$ is in the $\g_n$ neighborhood of $UCP_q(B^n)$.  On
the other hand, for any $\psi \in UCP_q(B^n)$ set $\var = \psi \circ
{\breve \s}$.  Then in the somewhat more complicated way given in
Section~\ref{sec13} we find that $UCP_q(B^n)$ is in as small a
neighborhood of $UCP_q(A)$ as desired if $n$ is sufficiently large.
\end{proof}

We remark that in section~$5$ of \cite{Krr} Kerr considers a weak form
of the Leibniz property which he calls ``$f$-Leibniz'' (for which he
comments that the corresponding distance may not satisfy the triangle
inequality).

In \cite{Lih2} Hanfeng Li introduced a quite flexible variant of
quantum Gromov--Hausdorff distance that in a suitable way uses the
Hausdorff distance between the unit $L$-balls of two quantum metric
spaces.  Li called this ``order-unit quantum Gromov--Hausdorff
distance''.  In \cite{KrL} Kerr and Li developed a matricial version
of Li's variant, which they called ``operator Gromov--Hausdorff
distance''.  They then show (theorem~$3.7$) that this version
coincides with Kerr's matricial quantum Gromov--Hausdorff distance.
It would be interesting to have a version of our complete proximity
above that is defined in terms of the unit $L$-balls, since it might
well have certain technical advantages similar to those possessed by
Li's order-unit Gromov--Hausdorff distance.

For the specific case of $C^*$-algebras, Li introduced \cite{Lih3} yet
another variant of quantum Gromov--Hausdorff distance that explicitly
uses the algebra multiplication.  He calls this ``$C^*$-algebraic
quantum Gromov--Hausdorff distance''.  It would be interesting to know
how this version relates to Leibniz seminorms and proximity.  We
should mention that in several places the later papers of Kerr and of
Li discussed in this section again consider the $f$-Leibniz property
that Kerr introduced in \cite{Krr}.

Hanfeng Li has pointed out to me that much the same arguments as
given in the last part of Section 6 
showing that  prox  is dominated by his $\mathrm{dist}_{nu}$,
also show that our ``complete proximity'' $\mathrm{prox}_s$ 
is dominated
by $\mathrm{dist}_{nu}$; and since, as mentioned in Section 6, 
the examples that have 
been studied so far for convergence for quantum Gromov-Hausdorff
distance all involve nuclear $C^*$-algebras, and convergence for them
holds for $\mathrm{dist}_{nu}$, this gives for them a proof of convergence
for $\mathrm{prox}_s$.

The papers discussed above all begin just with a $\Lip$-norm.  In a
different direction Wei Wu has defined and studied matricial Lipschitz
seminorms \cite{Wuw1,Wuw2,Wuw3}.  Again, we will not repeat here his
general definitions and results; rather we will only indicate somewhat
sketchily how they can be adapted to the context of the present paper,
I thank Wei Wu for answering several questions that I had about his
papers.

Let $G$ be a compact group equipped with a length function $\ell$, and
let $\a$ be an action of $G$ on a unital $C^*$-algebra $A$.  Then $G$
has an evident entry-wise action on $M_q(A)$ for each $q \in
\bZ_{>0}$, and we can then use $\ell$ to define a seminorm, $L^q$, on
each $M_q(A)$ as in Example 2.5.  This family of seminorms satisfies
Ruan-type axioms \cite{ER}, in particular, $L(T_{ij}) \le L^q(T)$ for
$T = \{T_{ij}\} \in M_q(A)$. Wu presents this family as one example of
what he calls a ``matrix Lipschitz seminorm'' on $A$.  It is a very
natural example, and it indicates how natural it is to consider matrix
Lipschitz seminorms quite generally.  However Wu does not make use of
the fact that each of the seminorms $L^q$ above is Leibniz (in fact,
strongly Leibniz), and he uses the bridge from \cite{R7}, which is not
Leibniz. 

For $A = C(G/H)$ and $B^n$ as earlier we denote the seminorms by
$L_A^q$ and $L_B^{n,q}$.
As Wu notes, the Berezin symbol map $\s^n$ gives, by entry-wise
application, a completely positive map from $M_q(B^n)$ to $M_q(A)$ for
each $q \in \bZ_{>0}$.  We denote these maps still by $\s^n$.  Much as
in Section~\ref{sec7} we can then define a seminorm on $M_q(A \oplus
B^n)$ by
\[
\|M_f \circ \s^n - \s^n \circ \L_T\|.
\]
But the analogue of the alternative description in terms of seminorms
$N_x$ given in Proposition~\ref{prop7.2} is now more complicated, and
so I have found it best just to work directly with the analogs of the
$N_x$'s.  Specifically, we write $\diag(\a_x(P^n))$ for the matrix in
$M_q(B^n)$ each of whose diagonal entries is $\a_x(P^n)$, with all
other entries being $0$.  For each $x \in G$ (or $G/H$) we set
\[
N_x^{n,q}(f, T) = \|\mathrm{diag}(\a_x(P))(f(x) \otimes I_n - T)\|
\]
for any $(f,T) \in M_q(A \oplus B^n)$. It is easily seen that
$N_x^{n,q}$ is strongly Leibniz. We then set
\[
N_{\s}^{n,q}(f,T) = \sup\{N_x^{n,q}(f,T): x \in G\}.
\]
Then we set
\[
N_{n,q}(f,T) = \g^{-1}N_{\s}^{n,q}(f,T),
\]
where $\g$ remains to be chosen for each $n$.  Finally we set
\[
L_{n,q}(f,T) = L_A^q(f) \vee L_B^{n,q}(T) \vee N_{n,q}(f,T) \vee
N_{n,q}^*(f,T).
\]
It is easily verified that the family $\{L_{n,q}\}$ is a ``matrix
Lipschitz seminorm'' as defined in definition~$3.1$ of \cite{Wuw3}.
We would like to choose $\g$ in such a way that the quotients of
$L_{n,q}$ on $M_q(A)$ and $M_q(B^n)$ are $L_A^q$ and $L_B^{n,q}$.

We consider the quotient on $M_q(A)$ first.  We note, as does Wu, that
${\breve \s}^n$ gives, by entry-wise application, a unital completely
positive map from $M_q(A)$ to $M_q(B^n)$.  Given $f \in M_q(A)$, we
set $T = {\breve \s}_f^n$.  Then, much as in Section~\ref{sec8},
\[
N_x^{n,q}(f,T) = \left\|\left\{\a_x(P^n)(f_{ij}(x)I_n - d \int
f_{ij}(y)\a_y(P^n)dy\right\}\right\|,
\]
where $\{\cdot \}$ denotes a matrix.
As in Section~\ref{sec8}, the translation-invariance of $L_A^q$ and
the arbitrariness of $f$ permit us to consider just the case in which
$x = e$.  Then, with manipulations as in Section~\ref{sec8}, we see
that
\begin{align*}
N_e^{n,q}(f,T) &\le d \int \|\{f_{ij}(e) - f_{ij}(y)\}\|
\|\diag(P^n\a_y(P^n))\|dy \\
&\le L_A^q(f) \int \rho(e,y)d\|P^n\a_y(P^n)\|dy \\
&= L_A^q(f)\g_n^A,
\end{align*}
where $\g_n^A$ is defined at the beginning of Section~\ref{sec9}.
Thus if $\g \ge \g_n^A$ then the quotient of $L_{n,q}$ on $M_q(A)$
will be $L_A^q$, which is exactly the same condition as for the case
of $q = 1$ treated in Section~\ref{sec8}.

We now consider the quotient on $M_q(B^n)$.  Given $T \in M_q(B^n)$,
we set $f = \s_T^n$.  Then
\[
N_x^{n,q}(f,T) = \|\{\a_x(P)(\tr(\a_x(P)T_{ij})I_n-T_{ij})\}\|.
\]
I don't see a good way to estimate this except by the entry-wise
estimate
\begin{align*}
&\le q \sup_{i,j} \|\a_x(P)(\tr(\a_x(P)T_{ij})I_n - T_{ij})\| \\
&\le q \g_n^B \sup_{i,j} L_n^B(T_{ij}) \le q \g_n^B L_B^{n,q}(T),
\end{align*}
where $\g_n^B$ is defined at the beginning of Section~\ref{sec12}, and
where we have used the $\a$-invariance of $L_n^B$, and the fact that
for any $R \in M_q(B^n)$ we have $\|R\| \le q
\sup_{i,j}\{\|R_{ij}\|\}$.  (To see this latter, express $R$ as the
sum of the $q$ matrices whose only non-zero entries are the entries
$R_{ij}$ of $R$ for which $i-j$ is constant modulo $q$.)  Thus if $\g
\ge q \g_n^B$ then the quotient of $L_{n,q}$ on $M_q(B^n)$ will be
$L_B^{n,q}$.  The factor of $q$ in this estimate has the quite
undesirable effect that we seem not to be able to say that for a
sufficiently large $\g$ it is true that for all $q$ simultaneously the
quotient of $L_{n,q}$ on $M_q(B^n)$ is $L_B^{n,q}$.  Thus the family
$\{L_{n,q}\}$ can not be used to estimate the ``quantized
Gromov--Hausdorff distance'' defined by Wu in definition~$4.5$ of
\cite{Wuw3}.  But for fixed $q$ we will still have that $q\g_n^B$
converges to $0$ as $n \to \i$, and this may still be useful, for
instance in dealing with vector bundles along the lines discussed in
\cite{R17}.

According to Wu's definition of ``quantized Gromov--Hausdorff
distance'' we must now show that $UCP_q(A)$ and $UCP_q(B_n)$ are
within suitable neighborhoods of each other in $UCP_q(A \oplus B)$
(once we have chosen $\g \ge \g_n^A \vee q\g_n^B$).  Given $f \in
M_q(A)$ and $\var \in UCP_q(A)$ (which Wu denotes by $CS_q(A)$), let
$\<\<\var,f\>\>$ denote the element of $M_{q^2}(\bC)$ whose entries
are the $\var_{ij}(f_{kl})$'s.  (See $1.1.27$ of \cite{ER}.)
Equivalently, view $f$ as in $M_q \otimes A$, and let ${\tilde \var} =
I_q \otimes \var$ so that ${\tilde \var}: M_q \otimes A \to M_q
\otimes M_q$.  Then $\<\<\var,f\>\> = {\tilde \var}(f)$.  We can thus
use $L_A^q$ to define a metric, $D_{L_A^q}$, on $UCP_q(A)$, defined by
\[
D_{L_A^q}(\var_1,\var_2) = \sup\{\|\<\<\var_1,f\>\> -
\<\<\var_2,f\>\>\|: f \in M_q(A), L_A^q(f) \le 1\}.
\]
(See proposition~$3.1$ of \cite{Wuw2}.)  Wu shows that the topology on
$UCP_q(A)$ from the metric $D_{L_A^q}$ coincides with the point-norm
topology.  In the same way $L_B^{n,q}$ defines a metric on
$UCP_q(B^n)$, and $L_{n,q}$ defines a metric on $UCP_q(A \oplus B^n)$.
Furthermore, when we view $UCP_q(A)$ and $UCP_q(B^n)$ as subsets of
$UCP_q(A \oplus B^n)$, the restriction of $D_{L_{n,q}}$ to them will
agree with $D_{L_A^q}$ and $D_{L_B^{n,q}}$ if the quotients of
$L_{n,q}$ on $M_q(A)$ and $M_q(B^n)$ agree with $L_A^q$ and
$L_B^{n,q}$.  (See proposition~$3.6$ of \cite{Wuw3}.)

We now show that $UCP_q(A)$ is in a suitably small neighborhood of
$UCP_q(B^n)$ for $D_{L_n^q}$.

\begin{lemma}
\label{lem14.2}
For any $(f,T) \in M_q(A \oplus B^n)$ we have
\[
\|f - \s_T^n\| \le q N_{\s}^{n,q}(f,T).
\]
\end{lemma}

\begin{proof}
\begin{align*}
\|f-\s_T^n\| &= \sup_x\|\{f_{ij}(x) - \tr(\a_x(P)T_{ij})\}\| \\
&\le q \sup_{x,i,j} |\tr(\a_x(P)(f_{ij}(x)I_n - T_{ij}))| \\
&\le q \sup_{x,i,j}\|\a_x(P)(f_{ij}(x)I_n - T_{ij})\| \\
&\le q \sup_x\|\{\a_x(P)(f_{ij}(x)I_n - T_{ij})\}\| = q
N_{\s}^{n,q}(f,T).
\end{align*}
\end{proof}

We can now proceed much as in the first half of Wu's proof of
theorem~$8.6$ of \cite{Wuw3}. Let $q$ be fixed, and now set $\g_n =
\g_n^A \vee q\g_n^B$ in the definition of $L_{n,q}$, so that $L_{n,q}$
has the right quotients.  Let $\var \in UCP_q(A)$ be given.  Set $\psi
= \var \circ \s^n$, so that $\psi \in UCP_q(B^n)$.  Suppose that
$(f,T) \in M_q(A \oplus B^n)$ and that $L_n^q(f,T) \le 1$, so that
$N_{\s}^{n,q}(f,T) \le \g_n$.  Then by Lemma~\ref{lem14.2}
\begin{align*}
\|\<\<\var,f\>\> - \<\<\psi,T\>\>\| &= \|\<\<\var,f-\s_T^n\>\>\| \\
&\le \|f-\s_T^n\| \le q N_{\s}^{n,q}(f,T) \le q \g_n .
\end{align*}
Thus $UCP_q(A)$ is in the $q\g_n$-neighborhood of $UCP_q(B_n)$.  Since
$\g_n^A \vee q\g_n^B$ converges to $0$ as $n\to \i$ we can make
$q\g_n$ as small as desired by choosing $n$ large enough.

We now show that $UCP_q(B^n)$ is in a suitably small neighborhood of
$UCP_q(A)$.  We can proceed as in the second half of Wu's proof of his
theorem~$8.6$ of \cite{Wuw3}.  Let $\psi \in UCP_q(B^n)$ be given.
Set $\var = \psi \circ {\breve \s}^n$, so that $\var \in UCP_q(A)$.
For $L(f,T) \le 1$ as above we have, much as in the proof of
Proposition~\ref{prop8.5},
\begin{align*}
\|\<\<\var,f\>\> - \<\<\psi,T\>\>\| &= \|\<\<\psi,{\breve \s}_f^n -
T\>\>\| \\
&\le \|{\breve \s}_f^n-T\| \le \|{\breve \s}_f^n - {\breve
\s}^n(\s_T^n)\| + \|{\breve \s}^n(\s_T^n) - T\| \\
&\le \|f-\s_T^n\| + \|{\breve \s}^n(\s_T^n) - T\| \\
&\le q\g_n + \|{\breve \s}^n(\s_T^n) - T\|.
\end{align*}
We can deal with the second of these terms much as we do in
Section~\ref{sec13}, just as Wu does.  One then sees that for a given
$\e > 0$ one can (for a fixed $q$) choose $N$ large enough that
$UCP_q(A)$ and $UCP_q(B^n)$ are in each other's $\e$-neighborhood for
$n \ge N$.


\def\dbar{\leavevmode\hbox to 0pt{\hskip.2ex \accent"16\hss}d}
\providecommand{\bysame}{\leavevmode\hbox to3em{\hrulefill}\thinspace}
\providecommand{\MR}{\relax\ifhmode\unskip\space\fi MR }
\providecommand{\MRhref}[2]{%
  \href{http://www.ams.org/mathscinet-getitem?mr=#1}{#2}
}
\providecommand{\href}[2]{#2}



\end{document}